\documentclass[12pt]{amsart}
\usepackage{graphicx}
\usepackage{graphicx}
\usepackage{epsfig}
\usepackage{eucal}
\usepackage{tikz}
\usetikzlibrary{arrows}

\newtheorem{thm}{Theorem}[section]

\newtheorem{defn}[thm]{Definition}
\newtheorem{lemma}[thm]{Lemma}

\newtheorem{cor}[thm]{Corollary}
\newtheorem{remark}[thm]{Remark}
\newtheorem{example}[thm]{Example}

\usepackage{amsmath}
\usepackage{amsxtra}
\usepackage{amscd}
\usepackage{amsthm}
\usepackage{amsfonts}
\usepackage{amssymb}
\usepackage{eucal}

\newcommand{\bmb}{\left( \begin{array}{rr}}
\newcommand{\enm}{\end{array}\right)}

\newcommand{\g}{{\mathfrak{g}}}

\renewcommand{\sl}{{\mathfrak{sl}}}

\newcommand{\ch}{{\rm ch}}
\newcommand{\C}{{\mathbb C}}

\newcommand{\Z}{{\mathbb Z}}
\newcommand{\Q}{{\mathcal Q}}

\newcommand{\Hom}{{\rm Hom}}

\newcommand{\bm}{{\mathbf m}}
\newcommand{\bn}{{\mathbf n}}

\newcommand{\bx}{{\mathbf x}}

\newcommand{\bu}{{\mathbf u}}

\newcommand{\bz}{{\mathbf z}}

\newcommand{\al}{{\alpha}}

\newcommand{\M}{\mathcal{M}}

\newcommand{\Zv}{{\mathbb Z}_v}
\numberwithin{equation}{section}

\usepackage{color}
\begin{document}

\title{Difference equations for graded characters from quantum cluster algebra}
\author{Philippe Di Francesco} 
\address{PDF: Department of Mathematics, University of Illinois MC-382, Urbana, IL 61821, U.S.A. e-mail: philippe@illinois.edu}
\author{Rinat Kedem}
\address{RK: Department of Mathematics, University of Illinois MC-382, Urbana, IL 61821, U.S.A. e-mail: rinat@illinois.edu}
\date{\today}
\begin{abstract}
We introduce a new set of $q$-difference operators acting as raising operators on a family of symmetric polynomials which are characters of graded tensor products of current algebra $\g[u]$ KR-modules \cite{FL} for $\g=A_r$. 
These operators are generalizations of the Kirillov-Noumi \cite{kinoum} Macdonald raising operators, in the dual $q$-Whittaker limit $t\to\infty$. They form a representation of the quantum $Q$-system of type $A$ \cite{qKR}.
This system is a subalgebra of a quantum cluster algebra, and is also a discrete integrable system whose
conserved quantities, analogous to the Casimirs of $U_q({\mathfrak sl}_{r+1})$, act as difference operators on the above family of symmetric polynomials. The characters
in the special case of products of fundamental modules are class I $q$-Whittaker functions, or characters of level-1 Demazure modules or Weyl modules. The action of the  conserved quantities on these characters gives the difference quantum Toda equations \cite{Etingof}. We obtain a generalization of the latter for arbitrary tensor products of KR-modules.

\end{abstract}

\maketitle
\date{\today}
\section{Introduction}

We consider the set of all symmetric polynomials in $r+1$ variables with coefficients in $\Z[q]$ which arise as graded characters of tensor products of $\g[u]$-modules, as defined by \cite{FL}. Here, $\g=sl_{r+1}$, $\g[u]$ are polynomials in $u$ with coefficients in $\g$, and we restrict our attention to tensor products of Kirillov-Reshetikhin (KR) modules \cite{ChariMoura}. In this case these are simply finite-dimensional, irreducible $\g$-modules with highest weights which are multiples of a fundamental weight, with an induced $\g[u]$-action.

The definition of the grading in \cite{FL} was motivated by the action of the affine algebra $\widehat{\g}$ on conformal blocks in WZW conformal field theory. It is described entirely in terms of the action of $\g[u]$ on tensor products of finite-dimensional modules. In special cases, the graded tensor product is a Demazure module of $\g[u]$ \cite{FourierLittelmann} or a Weyl module \cite{ChariLoktev}. In special, stabilized limits, their characters coincide with the characters of affine algebra modules.

It was conjectured in \cite{FL}, and subsequently proved \cite{ArdonneKedem,DFK08}, that the grading is equivalent to that of the quantum algebra action on the analogous tensor product in the crystal limit, related to the (finite) quantum spin chain whose Hilbert space is the tensor product of the associated quantum group modules \cite{OSS}. In the case of $\sl_n$, the coefficients of the Schur functions in the expansion of the graded characters, the graded multiplicities, are generalized Kostka polynomials \cite{SW}.

In our previous work \cite{qKR}, we gave a simple formulation of these characters in terms of generators of a non-commutative algebra, the quantum $Q$-system, which is a subalgebra of a quantum cluster algebra \cite{BernZel}. The graded tensor product mutiplicities are computed as linear functionals of the corresponding monomial in generators of the quantum cluster algebra (see Theorem \ref{CTId} below).

In this paper, we turn to the consequences of this formulation, and consider the action of two natural operators on this product: Multiplication by generators of the algebra, and the action of the conserved quantities of the quantum $Q$-system considered as a non-commutative integrable, discrete evolution. We find that the action by the conserved quantities is a generalization, in the case when the tensor product includes non-fundamental modules, of the $q$-deformed quantum Toda Hamiltonians \cite{Etingof}. The action by generators of the algebra is given by $q$-difference operators, which in the case of fundamental modules is the dual $q$-Whittaker limit of the Macdonald raising operators of Kirillov and Noumi \cite{kinoum}.

Let us summarize our results briefly. Let $\bn=\{n_\ell^{(\al)}, 1\leq \ell\leq k, \al\in[1,r]\}$ be a set of non-negative integers, where $k$ is some positive integer which we call the {\em level}. The tensor product 
of $\g$-modules 
$$\mathcal{M}_\bn=\underset{1\leq\al\leq r}{\otimes}\underset{1\leq \ell \leq k}{\otimes} V(\ell\omega_\al)^{\otimes n_\ell^{(\al)}},$$ 
where $\omega_\al$ are the fundamental weights and $V(\lambda)$ is the irreducible $\g$-module with highest weight $\lambda$, can be endowed with the structure of a $\g[t]$-module and a $\g$-equivariant grading \cite{FL}. The graded components $\M_{\bn}[j]$ are $\g$-modules, and the graded character of $\mathcal M_\bn$ are defined as
$$\chi_{\bn}(q,\bz) = \sum_{j\geq 0} q^j \ch_\bz \M_{\bn}[j]$$
where $\ch_\bz$ is the classical character, expanded in terms of Schur functions of $\bz=(z_1,...,z_{r+1})$. 

The first operator which acts on the characters is a generalized $q$-deformed Toda operator:
\vskip.1in
\noindent{\bf Theorem \ref{difchar}.}{\it Let $k\geq 1$ and let $n_k^{(\al)}$ and $n_{k-1}^{(\al)}$ be greater than or equal to $1-\delta_{k,1}$ for all $\al$. 
The graded characters $\chi_\bn\overset{\rm def}{=} \chi_\bn(q^{-1},\bz)$
satisfy the following difference equation:
\begin{eqnarray*}
&&\sum_{\al=1}^{r+1} 
\chi_{\bn+\epsilon_{\al-1,k-1}-\epsilon_{\al,k-1}+\epsilon_{\al,k}-\epsilon_{\al-1,k}} \nonumber \\
&&-\sum_{\al=1}^r q^{k-1-\sum_{i=1}^k i n_i^{(\al)}}
\chi_{\bn+\epsilon_{\al-1,k-1}-\epsilon_{\al,k-1}+\epsilon_{\al+1,k}-\epsilon_{\al,k}}= e_1(\bz) \, \chi_\bn
\end{eqnarray*}
where the vector $\epsilon_{\al,i}$ is defined so that $(\epsilon_{\al,i})_j^{(\beta)}=\delta_{\beta,\al}\delta_{j,i}$,
and $e_1(\bz)=z_1+z_2+\cdots +z_{r+1}$.
}
\vskip.1in

In the special case when $k=1$, we show in Section \ref{todawhit} that the graded characters are special $q$-Whittaker functions, and the difference equations in Theorem \ref{difchar} are the 
$U_q(\sl_{r+1})$ difference Toda equation \cite{Etingof}.

\medskip

The second operator which acts on the graded characters adds a factor to the tensor product. This too can be expressed as the action of a $q$-difference operator on $\chi_\bn(q^{-1};\bz)$, which we call a raising operator. As a result, there is an expression for the solutions of the difference equations of Theorem \ref{difchar} as the product of raising operators on the constant function 1. When $k=1$, the difference operators are a dual $q$-Whittaker limit degeneration of
the difference Macdonald raising operators of Kirillov and Noumi \cite{kinoum}.

Let $q=v^{-r-1}$ and let $D_i(z_1,...,z_{r+1})=(v z_1,...,q v z_i, ..., v z_{r+1})$.
Given a subset $I\subset [1,r+1]$, let $\overline I$ be its complement, and denote
\begin{equation}
z_I=\prod_{i\in I} z_i, \quad D_I=\prod_{i\in I} D_i,\quad {\rm  and}\quad 
a_I(\bz)=\prod_{i\in I\atop j\in \bar I}\frac{z_i}{z_i-z_j}.
\label{productnotation}
\end{equation}
Define the difference operators  ${\mathcal D}_{\al,n}$ acting on the space of 
Laurent polynomials in $\bz=(z_1,...,z_{r+1})$ with coefficients in $\Z[v,v^{-1}]$ as follows:
\begin{equation}
{\mathcal D}_{\al,n}=v^{-\frac{\Lambda_{\al,\al}}{2}n-\sum_{\beta=1}^r \Lambda_{\al,\beta}}
\sum_{I\subset [1,r+1]\atop |I|=\al } (z_I)^n a_I(\bz)\, D_I \qquad \al\in[0,r+1],\ n\in \Z.
\end{equation}
The matrix $\Lambda$ is given in Equation \eqref{lambdadef}.

The crucial property satisfied by the operators ${\mathcal D}_{\al,k}$ is that they obey the dual quantum $Q$-system relations
(see Theorem \ref{repsQ}).
The main result of Section \ref{demazure}, relying on this observation, is the following expression for the graded characters:
\vskip.1in
\noindent{\bf Theorem \ref{almost}.}{\it
The graded characters for ${\mathfrak sl}_{r+1}$ at level $k$ are given by:
\begin{eqnarray*}
\chi_\bn(q^{-1},\bz)&= &v^{\frac{1}{2} \sum_{i,j,\al,\beta} n_i^{(\al)}{\rm Min}(i,j)\Lambda_{\al,\beta}n_j^{(\beta)}+\sum_{i,\al,\beta} n_i^{(\al)}\Lambda_{\al,\beta}+\frac{1}{2}\sum_\al \Lambda_{\al,\al}+\sum_{\al<\beta}\Lambda_{\al,\beta}} \nonumber \\
&& \quad \times  \prod_{\al=1}^r({\mathcal D}_{\al,k})^{n_k^{(\al)}}\prod_{\al=1}^r({\mathcal D}_{\al,k-1})^{n_{k-1}^{(\al)}}
\cdots
\prod_{\al=1}^r({\mathcal D}_{\al,1})^{n_1^{(\al)}}\, 1
\end{eqnarray*}
}

\medskip

The paper is organized as follows. Section \ref{notadefs} gathers notations and definitions.
In Section \ref{consdiff}, we derive a set of difference equations for the characters, by using the 
conserved quantities of the quantum $Q$-system. The particular case of characters for fundamental KR-modules
is addressed in Section \ref{todawhit}, where they are identified with $q$-Whittaker functions, by
identifying the difference equation they obey with the $q$-deformed quantum Toda equation. In Section \ref{demazure},
we present the general construction of the characters by iterated action of $q$-difference operators on the constant $1$.
This is proved by showing that the latter satisfy the dual quantum $Q$-system relations and realize the action of the
quantum $Q$-system generators on characters by adding one extra factor in the tensor product. Details of the proofs 
are given in Appendices A and B.
\vskip.2in

\noindent{\bf Acknowledgments.} We thank O.Babelon, M.Bergvelt, A.Borodin, I. Cherednik, I.Corwin, V. Pasquier, and S.Shakirov for discussions at various stages of this work. R.K.'s research is supported by NSF grant DMS-1404988. P.D.F. is supported by the NSF grant DMS-1301636 and the Morris and Gertrude Fine endowment. R.K. would like to thank the Institut de Physique Th\'eorique (IPhT) of Saclay, France, for hospitality during various stages of this work.

\section{Notations and definitions}\label{notadefs}
The starting point for the results of this paper is the algebra called the quantum Q-system. This is a subalgebra of the quantum cluster algebra  \cite{BernZel} associated with the $A_r$ Q-system \cite{Ke,DFKnoncom,qKR}. The main results of the paper are representations of this algebra and its conserved quantities acting as difference operators on the space of symmetric Laurent polynomials in $r+1$ variables.

\subsection{Quantum $Q$-system}

The $A_r$ quantum $Q$-system is a non-commuative algebra 
$\mathcal A$ generated by 
invertible elements $\{\Q_{\al,k}: \al\in[1,r],k\in\Z\}$ subject to the quantum $Q$-system relations:
\begin{equation}\label{qQsys}
v^{\Lambda_{\al,\al} }\Q_{\al,k+1}\Q_{\al,k-1} = \Q_{\al,k}^2- \Q_{\al+1,k}\Q_{\al-1,k},\qquad \Q_{0,k}=\Q_{r+1,k}=1,
\end{equation}
as well as the commutation relations
\begin{equation}\label{commualbet}
\Q_{\al,k}\Q_{\beta,k'} = v^{\Lambda_{\al,\beta}(k'-k) }\Q_{\beta,k'}\Q_{\al,k}\qquad (|k-k'|\leq |\al-\beta|+1).
\end{equation}
The matrix $\Lambda$ is proportional to the inverse of the Cartan matrix of $A_r$:
\begin{equation}\label{lambdadef}
\Lambda_{\al,\beta}={\rm Min(\al,\beta)}\left(r+1- {\rm Max}(\al,\beta)\right), \quad (\al,\beta\in [1,r]).\end{equation}
Here, $v$ is an invertible central element of the algebra.
Note that all the variables $\Q_{\al,k}$ for different $\al$ and fixed $k$ commute with each other. 

The algebra $\mathcal A$, which contains the inverses of all its generators, is finitely generated by any set of $2r$ generators and their inverses which belong to the same ``cluster" \cite{DFKnoncom}. For example, the set $\mathcal S_0=\{\Q_{\al,0},\Q_{\al,1}: \al\in [1,r]\}$. Since $\mathcal A$ is a subalgebra of a quantum cluster algebra, all other generators $\Q_{\al,k}$ with $k\in \Z$ are {\em Laurent polynomials} in the generating set.

In Section \ref{conserved}, we provide a brief review of the discrete integrable structure of the quantum $Q$-system.
Explicit solutions were worked out in detail in Ref. \cite{PDFqmult}.

\subsection{The constant term of elements in $\mathcal A$}
Starting with the quantum Q-system, we showed in \cite{qKR} that there is a linear functional from monomials in positive powers of the generators of $\mathcal A$ to characters of graded tensor products of KR-modules. In \cite{qKR} showed this for all simply-laced algebras, but in the current context, we concentrate on type $A$ exclusively.

Quantum cluster algebras have a Laurent property, generalizing the one for commutative cluster algebras \cite{FominZel}. As a consequence, denoting by $\Z_v:=\Z[v,v^{-1}]$:
\begin{lemma}\label{lopo}
Given a set of initial data $\mathcal S_0=\{\Q_{\al,0},\Q_{\al,1}: \al\in [1,r]\}$, any solution $\Q_{\al,k}$ of the quantum $Q$-system can be expressed as a Laurent polynomial of the elements of $\mathcal S_0$, with coefficients in $\Zv$.
\end{lemma}

Thus, any polynomial in the generators of $\mathcal A$ can be expressed as a Laurent polynomial in the elements of $\mathcal S_0$. Since these elements $q$-commute according to Equation \eqref{commualbet}, we can define a normal ordering for any monomial in the elements of $\mathcal S_0$ as follows.

\begin{defn} The {\em normal ordered expression} of a monomial in the generators in $\mathcal S_0$ is the expression obtained, using the commutation relations \eqref{commualbet}, when all the $\Q_{\al,0}$ are written to the left of all the $\Q_{\beta,1}$, for all $\al,\beta$.
\end{defn}
The normal ordering extends to any Laurent polynomial or series in the elements of $\mathcal S_0$.

Normal ordering is necessary in order to give a unique meaning to the evaluation of a polynomial in $\mathcal S_0$ at some central value of the subset $\{\Q_{\al,0}\}_\al$, because these generators do not commute with the generators  $\{\Q_{\al,1}\}_\al$.
The evaluation occurs only {\em after} normal ordering.

\begin{defn}\label{evalmap}
The  linear map 
$ev: \Zv[\{\Q_{\al,0}^{\pm1},\Q_{\al,1}^{\pm1}\}_\al]\to \Zv[\{\Q_{\al,1}^{\pm1}\}_\al] $ 
is given by (1) normal ordering the Laurent polynomial of the variables in $\mathcal S_0$, 
and then (2) setting $\Q_{\al,0}=1$ for all $\al$ in the normal-ordered expression.
\end{defn}
A closely related map is the following:
\begin{defn}\label{evalzero}
The map 
$ev_0: \Zv[\Q_{\al,0}^{\pm1},\Q_{\al,1}^{\pm1}]\to \Zv[\Q_{\al,1}^{\pm1}] $ 
is given by (1) normal ordering the Laurent polynomial of the variables in $\mathcal S_0$, 
and then (2) setting $\Q_{\al,0}=v^{-\sum_\beta \Lambda_{\al,\beta}}$ for all $\al$ in the normal-ordered expression.
\end{defn}
The two evaluation maps are related in the following manner:
\begin{lemma}\label{equivev}
For any Laurent polynomial $f\in \Zv[\Q_{\al,0}^{\pm1},\Q_{\al,1}^{\pm1}]$, we have:
$$ ev\left( \prod_{\beta=1}^r \Q_{\beta,1} \, f \right)= \left(\prod_{\beta=1}^r \Q_{\beta,1} \right)\, ev_0(f) $$
\end{lemma}
\begin{proof}
The commutation relations \eqref{commualbet} imply:
$\left(\prod_{\beta=1}^r \Q_{\beta,1} \right) \Q_{\al,0}=v^{-\sum_\beta \Lambda_{\al,\beta}}\Q_{\al,0}\prod_{\beta=1}^r \Q_{\beta,1}$.
\end{proof}

We also extend the notion of the evaluation map to any Laurent polynomials or series in the generators of $\mathcal S_0$.

There is a stronger version of Lemma \ref{lopo} in the case of the quantum Q-system, which is a {\em polynomiality} property due to the specific form of the quantum $Q$-system (see Corollary 5.13 of \cite{qKR}):
\begin{lemma}\label{poly}
Let $f$ be a polynomial of the variables $\{\Q_{\al,k}, \al\in [1,r], k\geq 1\}$, obeying the quantum $Q$-system. Then 
$ev_0(f)\in \Zv[\{\Q_{\beta,1},{\beta\in [1,r]}\}]$, namely it is a {\em polynomial} of 
the variables $\{\Q_{\beta,1}\}_{\beta\in [1,r]}$, with coefficients which are Laurent polynomials in $v$. 
\end{lemma}
As a consequence of Lemma \ref{equivev}, we can restate the polynomiality property as:
$ev(\prod_{\beta=1}^r \Q_{\beta,1} \, f)\in \left(\prod_{\beta=1}^r \Q_{\beta,1}\right) \Zv[\{\Q_{\beta,1}\}_{\beta\in [1,r]}]$ 
is a polynomial of the variables $\{\Q_{\beta,1}\}_{\beta\in [1,r]}$,
which is a multiple of $\prod_{\beta=1}^r \Q_{\beta,1}$.

The following two definitions concern the evaluation of Laurent series, which we use only where these evaluations converge.
\begin{defn}
Given a Laurent series $f$ in $\{\Q_{\al,1}^{-1}, \al\in[1,r]\}$ with coefficients in $\Zv$, the map CT(f) is the constant term in $\Q_{\al,1}$ for all $\al$.
\end{defn}

\begin{defn} \label{phidef}
Given a Laurent series $f$ in $\{\Q_{\al,1}^{-1}, \al\in[1,r]\}$ with coefficients in 
$\Zv[\Q_{\al,0}^{\pm1}]$,  
we define the linear map $\phi=CT\circ ev$ sends such a series to an element in $\Zv$ by first evaluating the normal ordered expression of $f$ at all $\Q_{\al,0}=1$, and then extracting the constant term in all $\Q_{\al,1}$.
\end{defn}

\subsection{Constant term identity for graded tensor product multiplicities}

The main tool in the paper \cite{qKR} was an expression of any graded tensor product multiplicity as the constant term of a corresponding monomial in the generators of $\mathcal A$. 
Fix a dominant $\g$-weight $\lambda = \sum_\al \ell_\al \omega_\al\in P^+$ and a set of non-negative integers $\bn=\{n_{i}^{(\al)}\}$ as before.
The graded tensor product multiplicities
$$M_{\bn,\lambda}(q):=\sum_j q^j \dim(\Hom_\g (\M_\bn[j], V(\lambda))$$
can be expressed as follows:

\begin{thm}\label{CTId}\cite{qKR}
The Graded multiplicities of the irreducible components in the level $k$ $M$-sum formula of \cite{KR} can be expressed as
\begin{eqnarray}\label{gtensopro}
M_{\bn,\lambda}(q^{-1}) &=&v^{\sum_{\al,\beta,i}n^{(\al)}_{i}\Lambda_{\al,\beta}+
{1\over 2}(\sum_\al \ell_\al \Lambda_{\al,\al}+\sum_{i,j,\al,\beta} n_i^{(\al)}{\rm Min}(i,j)\Lambda_{\al,\beta}n_j^{(\beta)}) }\nonumber \\
&&\qquad \phi \left(\left(\prod_{\al=1}^r \Q_{\al,1}\Q_{\al,0}^{-1}\right)
\left(\prod_{i=1}^k \prod_{\al=1}^r (\Q_{\al,i})^{n^{(\al)}_{i}} \right) \prod_{\al=1}^r \lim_{k\to\infty} 
(\Q_{\al,k}\Q_{\al,k+1}^{-1})^{\ell_\al+1} \right),
\end{eqnarray}
where $\Q_{\al,n}$ are the generators of $\mathcal A$, and 
\begin{equation}\label{ttoq}
q=v^{-r-1}
\end{equation}
is central.
\end{thm}

We will use both variables, $v$ and $q$ throughout the paper. One can show that the polynomial $M_{\bn,\lambda}$ is a function of $q$ only.

Note that in Equation \eqref{gtensopro}, the monomial involving the $\Q_{\al,i}$ becomes a {\em polynomial} of the initial data variables $\Q_{\al,1}$ after the evaluation step of Definition \ref{evalmap}, as a consequence of the polynomiality Lemma \ref{poly}, and the obvious property that $ev(f g)=ev(ev(f) g)$. However 
the ``tails" $\lim_{k\to\infty} 
\Q_{\al,k}\Q_{\al,k+1}^{-1}$ are Laurent series in $Q_{\al,1}^{-1}$, for $\al\in [1,r]$. 
Therefore, the constant terms pick only finitely many contributions.

Finally, \eqref{gtensopro} may be translated into an analogous expression for
the graded character, by using:
\begin{equation}\label{char}
\chi_\bn(q^{-1};\bz)=\sum_{\lambda\in P^+} M_{\bn,\lambda}(q^{-1})\, s_\lambda(\bz) 
\end{equation}
where $s_\lambda(\bz)={\rm ch}_\bz\, V(\lambda)$ is a Schur function, $P^+$ the set of positive weights of $\g$.

\section{From conserved quantities to difference equations}
\label{consdiff}
In this section, we derive difference equations for the graded characters $\chi_\bn(q;\bz)$, by using the explicit conserved quantities of the quantum $Q$-system \eqref{qQsys}, viewed as a discrete, integrable evolution equation in $\mathcal A$ \cite{DFKnoncom}.

\subsection{Conserved quantities of the quantum $Q$-system}
\label{conserved}
Discrete integrability of the system \eqref{qQsys} means that there are $r$ discrete, algebraically
independent conserved quantities, $C_m$, $m=1,2,...,r$ which commute with each other. Moreover, these are coefficients of a linear recursion relation satisfied by the generators
$\{\Q_{1,k}\}$. 

Let us recall the explicit formulas for the conserved quantities \cite{DFKnoncom}.
For each $n\in \Z$, define the {\em weights} $y_i(n)\in \mathcal A$ as the following ordered monomials:
\begin{eqnarray}\label{weights}
y_{2\al-1}(n)&=&\Q_{\al,n+1}\Q_{\al-1,n+1}^{-1}\Q_{\al,n}^{-1}\Q_{\al-1,n} ,\qquad (\al=1,2,...,r+1)\label{yodd}; \\
y_{2\al}(n)&=&-\Q_{\al+1,n+1}\Q_{\al,n+1}^{-1}\Q_{\al,n}^{-1}\Q_{\al-1,n}, \qquad (\al=1,2,...,r) \label{yeven}.
\end{eqnarray}

\begin{thm}{\cite{DFKnoncom}}\label{consQ} Modulo the quantum Q-system, the following elements of $\mathcal A$ are independent of $n$:
\begin{equation}
C_m:=\sum_{1\leq i_1<i_2\cdots <i_m\leq 2r+1\atop {i_\ell<i_{\ell+1}-1\ {\rm if}\ i_\ell \ {\rm odd}\atop
i_\ell<i_{\ell+1}-2\ {\rm if}\ i_\ell \ {\rm even} }} y_{i_m}(n)y_{i_{m-1}}(n)\cdots y_{i_1}(n)\qquad (m=0,1,...,r+1).
\label{conservedm}
\end{equation}
\end{thm}
In particular, $C_0=1$, and $C_{r+1}=v^{\frac{r(r+1)}{2}}$, and
$C_1,...,C_r$ are algebraically independent, commuting elements of $\mathcal A$, 
which are the conserved quantities of the quantum $Q$-system.
Of particular interest is the first conserved quantity:
\begin{lemma}
The conserved quantity $C_1$ is
\begin{equation}
C_1=\sum_{\al=1}^{r+1} \Q_{\al,n}^{-1}\Q_{\al-1,n}\left(v^r \Q_{\al,n+1}\Q_{\al-1,n+1}^{-1}-
v^{-1} \Q_{\al+1,n+1}\Q_{\al,n+1}^{-1}\right). \label{finacone}.
\end{equation}
\end{lemma}
\begin{proof}
Fix $n$ and use the expression for the weights \eqref{weights} as well as the commutation relations \eqref{commualbet}:
\begin{eqnarray}
C_1&=&\sum_{\al=1}^{r+1} y_{2\al-1}(n)+\sum_{\al=1}^r y_{2\al}(n)\nonumber \\
&=&\sum_{\al=1}^{r+1} v^{\Lambda_{\al,\al}-2\Lambda_{\al,\al-1}+\Lambda_{\al-1,\al-1}}
\Q_{\al,n}^{-1}\Q_{\al-1,n}\Q_{\al,n+1}\Q_{\al-1,n+1}^{-1}\nonumber \\
&&-\sum_{\al=1}^r 
v^{\Lambda_{\al,\al+1}-\Lambda_{\al-1,\al+1}-\Lambda_{\al,\al}+\Lambda_{\al-1,\al}} 
\Q_{\al,n}^{-1}\Q_{\al-1,n}\Q_{\al+1,n+1}\Q_{\al,n+1}^{-1} \nonumber .\\
\end{eqnarray}
The Lemma follows from the identities:
\begin{equation*}
\Lambda_{\al,\al}-2\Lambda_{\al,\al-1}+\Lambda_{\al-1,\al-1}=r,\quad 
\Lambda_{\al,\al+1}-\Lambda_{\al-1,\al+1}-\Lambda_{\al,\al}+\Lambda_{\al-1,\al}=-1.
\end{equation*}
\end{proof}

Define the elements of $\mathcal A$
\begin{equation}\label{tailbits}
\theta_{\al,k}=\Q_{\al,k}\Q_{\al,k+1}^{-1},\qquad \xi_{\al,k}=v^{\frac{\Lambda_{\al,\al}}{2}}\theta_{\al,k}.
\end{equation}
In \cite{qKR}, we showed that $\xi_{\al,k}$ is expressible as a formal power series of the variables $\Q_{\al,1}^{-1}$ with no constant term, with coefficients
which are Laurent polynomials of the $\Q_{\beta,0}$'s, and that the limit $k\to \infty$ exists. We denote it by  
\begin{equation}\label{xialdef}\xi_\al=\lim_{k\to\infty} \xi_{\al,k}.
\end{equation}

As the conserved quantities $C_m$ are independent of $n$, they may be evaluated in the limit $n\to \infty$.
We have:
\begin{lemma}\label{limlem}
\begin{equation}\label{yinf}y_{2\al}:=\lim_{n\to\infty} y_{2\al}(n)=0 \quad {\rm and}\quad 
y_{2\al-1}:=\lim_{n\to\infty} y_{2\al-1}(n)=v^{\frac{r}{2}}\, \xi_{\al-1}\xi_\al^{-1} 
\end{equation}
Moreover, the odd variables $y_1,y_3,...,y_{2r+1}$ commute among themselves, and $C_m$ is their $m$-th 
elementary symmetric function:
\begin{equation}C_m=e_m(y_1,y_3,...,y_{2r+1})
=v^{mr/2}\, e_m(\xi_0\xi_1^{-1},\xi_2\xi_3^{-1},...,\xi_{2r}\xi_{2r+1}^{-1})\label{usefulxi}
\end{equation}
\end{lemma}

\begin{proof}
Using the commutation relations \eqref{commualbet}, we note that 
\begin{eqnarray*}y_{2\al-1}(n)&=&v^{\frac{r}{2}}\, \xi_{\al,n}^{-1}\, \xi_{\al-1,n} \\
y_{2\al}(n)&=& v^{\frac{r}{2}}\, \xi_{\al+1,n}^{-1}\, \xi_{\al,n} (\xi_{\al,n}\xi_{\al,n-1}^{-1}-1)
\end{eqnarray*}
where in the last line we have used the quantum $Q$-system relation to rewrite
$$-\Q_{\al+1,n}\Q_{\al-1,n}\Q_{\al,n}^{-2}=1-v^{\Lambda_{\al,\al}}\Q_{\al,n+1}\Q_{\al-1,n-1}\Q_{\al,n}^{-2}=
1-\xi_{\al,n}\xi_{\al,n-1}^{-1}$$
The limits \eqref{yinf} follow from \eqref{xialdef}. Finally the commutations follow from \eqref{commualbet},
and \eqref{conservedm} clearly reduces to the $m$-th elementary symmetric function of the odd $y$'s.
\end{proof}

\subsection{Generating series}
\label{genef}

We define generating series for the characters \eqref{char}.
First, define $\tau(\bz)=\tau(z_1,...,z_{r+1})$ as
\begin{eqnarray}\label{taudef}
\tau(\bz)&:=&v^{{1\over 2}\sum_\al \Lambda_{\al,\al}+\sum_{\al<\beta}\Lambda_{\al,\beta}}\sum_{\lambda\in P^+} 
\prod_{\al=1}^r (\xi_\al)^{\ell_\al+1} s_\lambda(\bz) \nonumber \\
&=& v^{\frac{r+1}{4}{r+2\choose 3}} \sum_{\lambda\in P^+} 
\prod_{\al=1}^r (\xi_\al)^{\ell_\al+1} s_\lambda(\bz),\end{eqnarray}
with $\xi_\al$ as in \eqref{xialdef}.
Here, $P^+$ is the set of dominant integral weights of $\sl_{r+1}$, whereas $s_\lambda(\bz)$ is the Schur function parameterized by partitions $\lambda$ of length $r+1$ or less, with the usual correspondence between $\sl_{r+1}$ weights and the set of such partitions, that is, $\ell_\al=\lambda_\al-\lambda_{\al+1}$.

Fix $k\geq 1$ and define the generating series of characters as a series in the indeterminates $\bu=\{u_{\al,i}: \al\in [1,r], i\in [1,k]\}$:
\begin{equation}\label{defGky}
G^{(k)}(\bu)=\phi\left(
\left(\prod_{\al=1}^r \Q_{\al,1} \right)\left(\prod_{i=1}^k\left( \prod_{\al=1}^r \frac{1}{1-u_{\al,i}\Q_{\al,i} }\right) \right)
\tau(\bz) \right).
\end{equation}
Here, each rational function is defined to be a series in the variables $u_{\al,i}$, and the product over $i$ is ordered from left to right. 

The coefficient of $\prod_{\al,i}u_{\al,i}^{n_i^{(\al)}}$ in the formal series expansion of $G^{(k)}(\bu)$ is defined to be $G_\bn^{(k)}$, with $\bn=\{n_{\al,i}\}_{\al\in [1,n];i\in [1,k]}$:
\begin{equation}\label{defGkn}G^{(k)}_\bn=\phi\left(
\left(\prod_{\al=1}^r \Q_{\al,1} \right)\left(\prod_{i=1}^k\left( \prod_{\al=1}^r \Q_{\al,i}^{n_i^{(\al)}}\right) \right)
\tau(\bz) \right)
\end{equation}
with the function $\phi$  defined in \ref{phidef}.
The normalization of $\tau(\bz)$ in \eqref{taudef} is chosen so that $G_0^{(1)}=G^{(1)}(0)=1$.

Comparing with Equation \eqref{gtensopro}, using the commutation relations between 
$\Q_{\al,0}$ and $\Q_{\beta,1}$, we see that these coefficients are the renormalized characters of Equation \eqref{char}:
\begin{equation}\label{chitoG}
\chi_{\bn}(q^{-1},\bz)=v^{\sum_{\al,\beta,i}n_i^{(\al)}\Lambda_{\al,\beta}+
{1\over 2}\bn\cdot (\Lambda\otimes A)\bn}\,G^{(k)}_\bn(\bz).
\end{equation}

\subsection{The action of the conserved quantities at infinity}
\label{Contau}

We have the following theorem for the action of the conserved quantities $C_m$ on $\tau(\bz). $
Let $e_m(\bz)$ denote the $m$th elementary symmetric function in the $r+1$ variables $z_1,...,z_{r+1}$. Then
\begin{thm}
The conserved quantities \eqref{conservedm} of the $A_r$ quantum $Q$-system act on the function $\tau(\bz)$ as:
\begin{equation}\label{actioncm}
C_m\, \tau(\bz)=v^{\frac{m r}{2}}\, e_m(\bz) \, \tau(\bz)+ R_m(\bz)
\end{equation}
where $R_m(\bz)$ is a sum of power series of the $\xi_\al$, with each of the summands independent 
of at least one of the $\{\xi_\al\}_{\al\in [1,r]}$.
\end{thm}
\begin{proof}
Recall the expression \eqref{taudef} for $\tau(\bz)$.
Using \eqref{usefulxi}, we write explicitly:
\begin{eqnarray*}
v^{-\frac{m r}{2}}C_m\, v^{-\frac{r+1}{4}{r+2\choose 3}}\tau(\bz)&=& 
\sum_{\ell_1,..,\ell_r\geq 0} \prod_{\al=1}^r \xi_\al^{\ell_\al+1}\sum_{1\leq i_1<i_2<\cdots <i_m} 
s_{\ell_1,..\ell_{i_1-2},\ell_{i_1-1}-1,\ell_{i_1}+1,...,\ell_{i_m-1}-1,\ell_{i_m}+1,...\ell_r}(\bz) \\
&=&  \sum_{\ell_1,..,\ell_r} \prod_{\al=1}^r \xi_\al^{\ell_\al+1}\sum_{1\leq i_1<i_2<\cdots <i_m} 
s_{\lambda+\epsilon_{i_1}+\cdots +\epsilon_{i_m}}(\bz)+u_m(\bz)\\
&=&  \sum_{\ell_1,..,\ell_r} \prod_{\al=1}^r \xi_\al^{\ell_\al+1} e_m(\bz) \, s_\lambda(\bz) +u_m(\bz)
=e_m(\bz)\, v^{-\frac{r+1}{4}{r+2\choose 3}}\tau(\bz)+u_m(\bz)
\end{eqnarray*}
where $u_m(\bz)$ has the same property as $R_m(\bz)$ (it adds counterterms for all cases where some of the $\ell_{i_j}$'s vanish in the first term of the first line, and each such term is independent of the corresponding $\xi_{i_j}$). 
We have used the fact that $\epsilon_i=\omega_i-\omega_{i-1}$ correspond to 
$\ell_\al=\delta_{\al,i}-\delta_{\al,i-1}$ and transcribed the result using the Pieri rule for $SL_{r+1}$.
The theorem follows.
\end{proof}

We deduce the following:
\begin{cor}\label{conscor}
When evaluated inside the generating function \eqref{defGky}, each conserved quantity $C_m$ acts on $\tau(\bz)$ as the scalar 
$v^{\frac{m r}{2}}e_m(\bz)$,
namely:
\begin{equation}\label{insert}
\phi\left(
\left(\prod_{\al=1}^r \Q_{\al,1} \right)\left(\prod_{i=1}^k\prod_{\al=1}^r \frac{1}{1-u_{\al,i}\Q_{\al,i}} \right) C_m
\tau(\bz)
\right)=v^{\frac{m r}{2}}e_m(\bz) \, G^{(k)}(\bu)
\end{equation}
\end{cor}
\begin{proof} Using \eqref{actioncm}, note that as
each summand of $R_m(\bz)$ has at least one missing $\xi_\al$, the corresponding constant term in $\Q_{\al,1}$ must vanish, 
as the rest of the power series only generates positive powers of $\Q_{\al,1}$, once left evaluated at $\Q_{\al,0}=1$.
\end{proof}

\subsection{Difference equations from conserved quantities}
\label{diff}

We now use a standard argument to reformulate the conserved quantities of the quantum $Q$-system into
difference equations for the quantities $G^{(k)}_\bn$ of \eqref{defGkn}.

\begin{thm}\label{difconeN}
For $k\geq 2$, the coefficients $G^{(k)}_\bn$, $\bn=\{n_i^{(\al)}\}_{i\in [1,k];\al\in [1,r]}$, $n_k^{(\al)},n_{k-1}^{(\al)}\geq 1$,
obey the following difference equation:
\begin{eqnarray}
&&v^{\frac{r}{2}} e_1(\bz) \, G^{(k)}_\bn
=\sum_{\al=1}^{r+1} v^{r+\sum_\beta (\Lambda_{\al,\beta}-\Lambda_{\al-1,\beta})n_k^{(\beta)}} \,
G_{\bn+\epsilon_{\al-1,k-1}-\epsilon_{\al,k-1}+\epsilon_{\al,k}-\epsilon_{\al-1,k}}^{(k)} \nonumber \\
&&\qquad\qquad -\sum_{\al=1}^r v^{-1+\sum_\beta (\Lambda_{\al,\beta}-\Lambda_{\al-1,\beta})n_k^{(\beta)}} \,
G_{\bn+\epsilon_{\al-1,k-1}-\epsilon_{\al,k-1}+\epsilon_{\al+1,k}-\epsilon_{\al,k}}^{(k)}
\label{diffeqk}
\end{eqnarray}
where we use the notation $\epsilon_{\beta,m}$ for the vector with entries $(\epsilon_{\beta,m})_i^{(\al)}=\delta_{\al,\beta}\delta_{i,m}$.
For $k=1$, the coefficients $G^{(1)}_\bn$, $\bn=\{n^{(\al)}\}_{\al\in [1,r]}$, $n^{(\al)}\geq 0$, obey the difference equation:
\begin{eqnarray}
&&v^{\frac{r}{2}} e_1(\bz) \, G^{(1)}_\bn=v^{r+\sum_\beta \Lambda_{1,\beta}(n^{(\beta)}+1)} \,
G_{\bn+\epsilon_{1}}^{(1)}\nonumber \\
&&\qquad\qquad +\sum_{\al=2}^{r+1} \left(v^{r+\sum_\beta (\Lambda_{\al,\beta}-\Lambda_{\al-1,\beta})(n^{(\beta)}+1)} 
-v^{-1+\sum_\beta (\Lambda_{\al+1,\beta}-\Lambda_{\al,\beta})(n^{(\beta)}+1)}\right) \,
G_{\bn+\epsilon_{\al}-\epsilon_{\al-1}}^{(1)}
\label{diffeqkone}
\end{eqnarray}
with the notation $\epsilon_\beta$ for the vector with entries $(\epsilon_\beta)^{(\al)}=\delta_{\al,\beta}$.
\end{thm}
\begin{proof}
We compute in two ways the quantity
$$ B=\phi\left(\left(\prod_{\al=1}^r \Q_{\al,1} \right)\left(\prod_{i=1}^k\prod_{\al=1}^r (\Q_{\al,i})^{n_i^{(\al)}} \right) C_1\tau(\bz)\right) $$
First, we find $B= v^{\frac{r}{2}} e_1(\bz) \, G^{(k)}_\bn$ by direct application of \eqref{insert}. Second, we use the expression
\eqref{finacone} with $n=k-1$ for $C_1$. Using the notation:
$$ \langle M \rangle_\bn=\phi\left(\left(\prod_{\al=1}^r \Q_{\al,1} \right)
\left(\prod_{i=1}^k\prod_{\al=1}^r (\Q_{\al,i})^{n_i^{(\al)}} \right)\, M \,\tau(\bz)\right)$$
for any Laurent monomial $M$ of the $Q$'s,
we have:
\begin{eqnarray*}
\langle \Q_{\al,k-1}^{-1}\Q_{\al-1,k-1}\Q_{\al,k}\Q_{\al-1,k}^{-1} \rangle_\bn 
&=&v^{\sum_\beta (\Lambda_{\al,\beta}-\Lambda_{\al-1,\beta})n_k^{(\beta)} } \langle 1 \rangle_{\bn+\epsilon_{\al-1,k-1}-\epsilon_{\al,k-1}+\epsilon_{\al,k}-\epsilon_{\al-1,k}}\\
\langle \Q_{\al,k-1}^{-1}\Q_{\al-1,k-1}\Q_{\al+1,k}\Q_{\al,k}^{-1} \rangle_\bn 
&=& v^{\sum_\beta (\Lambda_{\al,\beta}-\Lambda_{\al-1,\beta})n_k^{(\beta)} } 
\langle 1 \rangle_{\bn+\epsilon_{\al-1,k-1}-\epsilon_{\al,k-1}+\epsilon_{\al+1,k}-\epsilon_{\al,k}}
\end{eqnarray*}
The case $k=1$ must be treated separately, as the insertion of $\Q_{\al,0}$ amounts to a factor
$v^{-\sum_\beta \Lambda_{\al,\beta}}$, coming from commutation of 
$\Q_{\al,k-1}^{-1}\Q_{\al-1,k-1}$ through $\prod_\beta (\Q_{\beta,k})^{n_k^{(\beta)}}$.
The Theorem follows.
\end{proof}

\begin{example}
When $r=1$ (case of ${\mathfrak sl}_2$), we have for $n_i\equiv n_i^{(1)}$, and $\Lambda_{1,1}=1$:
$$v^{n_k+1}G^{(k)}_{n_1,...,n_{k-1}-1,n_k+1}+v^{1-n_k}G^{(k)}_{n_1,...,n_{k-1}+1,n_k-1}-v^{n_k-1} G^{(k)}_{n_1,...,n_{k-1}-1,n_k-1}=v^{\frac{1}{2}}(z+z^{-1})\, G^{(k)}_{n_1,...,n_{k-1},n_k}$$
with $z=z_1=z_2^{-1}$,
whereas for $k=1$, $n\equiv n_1$:
\begin{equation}\label{levonG}
v^{n+2}G^{(1)}_{n+1}+(v^{-n}-v^{n})G^{(1)}_{n-1}=v^{\frac{1}{2}}(z+z^{-1})\,G^{(1)}_n 
\end{equation}
\end{example}

More generally, repeating this with the other conserved quantities $C_m$, $m\geq 2$ leads to higher difference equations
of the form ${\mathcal D}^{(m)} G^{(k)}_\bn =v^{\frac{m r}{2}} e_m G^{(k)}_\bn$, where the difference operators ${\mathcal D}^{(m)}$
form a commuting family for $m=1,2,...,r$, and ${\mathcal D}^{(1)}$ acts on the function $G_\bn^{(k)}$ of $\bn$
via the l.h.s. of eq.\eqref{diffeqk}.

\begin{example}\label{oneexG}
When $r=2$ and $k=2$ (case of ${\mathfrak sl}_3$, level $2$), we have the following recursion relations in the 
variables $n^{(\al)}_1=n_\al$
and $n^{(\al)}_2=p_\al$, $\al=1,2$, obtained respectively by inserting the conserved quantities $C_1$ and $C_2$
of Example \ref{consatwo}:
\begin{eqnarray*}
&&G^{(2)}_{n_1-1,p_1;n_2+1,p_2}+v^{-3n_2}G^{(2)}_{n_1+1,p_1-1;n_2-1,p_2+1}
+v^{-3n_2-3p_2}G^{(2)}_{n_1,p_1+1;n_2,p_2-1} \\
&&\qquad -v^{-3}G^{(2)}_{n_1-1,p_1;n_2-1,p_2+1}-v^{-3-3n_2}G^{(2)}_{n_1+1,p_1-1;n_2,p_2-1}
= v^{-1-2n_2-p_2} e_1(\bz) \,G^{(2)}_{n_1,p_1;n_2,p_2}\\
&&G^{(2)}_{n_1,p_1-1;n_2,p_2+1}+v^{-3p_2}G^{(2)}_{n_1-1,p_1+1;n_2+1,p_2-1}
+v^{-3n_2-3p_2}G^{(2)}_{n_1+1,p_1;n_2-1,p_2} \\
&&\qquad-v^{-3}G^{(2)}_{n_1,p_1-1;n_2+1,p_2-1}-v^{-3-3p_2}G^{(2)}_{n_1-1,p_1+1;n_2-1,p_2}
= v^{-1-n_2-2p_2} e_2(\bz) \,G^{(2)}_{n_1,p_1;n_2,p_2}
\end{eqnarray*}
with $e_1(\bz)=z_1+z_2+z_3$ and $e_2(\bz)=z_1z_2+z_2z_3+z_1z_3$, $z_1z_2z_3=1$.
\end{example}

For later use, let us focus on the level $1$ higher difference equations obtained by inserting $C_m$, $m\in [1,r]$
into the bracket $\langle \cdots \rangle_\bn$ defined above. As apparent from eq.\eqref{diffeqkone}
of Theorem \ref{difconeN},
the difference equation for $G^{(1)}$ allows to express 
$G^{(1)}_{\bn+\epsilon_1}$ as a linear combination
of the shifted functions $G^{(1)}_{\bn+\epsilon_{\al+1}-\epsilon_\al}$, $\al=1,2,...,r$, as well as $G^{(1)}_{\bn}$.
Similarly, due to the form of the conserved quantities as functions of the $\Q_{\al,n}$'s, the level 1 $C_m$ difference
equation allows to express $G^{(1)}_{\bn+\epsilon_m}$ as a linear combination of shifted functions 
of the form: $G^{(1)}_{\bn+\sum_{1\leq i\leq m}\epsilon_{\al_i+1}-\epsilon_{\al_i}}$ with 
$1\leq \al_{1}< \cdots < \al_m \leq r$, as well as $G^{(1)}_{\bn}$. Combining all the equations for $m=1,2,...,r$
provides therefore a recursive method for computing all $G_{\bn}^{(1)}$. Indeed, defining $\sigma(\bn)=\sum_{\al} n^{(\al)}$, we see that each equation is a three term recursion in the variable $\sigma(\bn)$, as the term $\bn+\epsilon_m$ has a value of $\sigma$ 1 or 2 larger than all other terms. If we know all the values of $G^{(1)}_\bn$ for 
$\sigma(\bn)\leq N$, we therefore deduce $G^{(1)}_\bn$ for all values $\sigma(\bn)=N+1$. We have the following:

\begin{thm}\label{oneunique}
The difference equations obtained by inserting $C_m$, $m=1,2,...,r$ at level 1 determine the functions $G^{(1)}_\bn$ 
uniquely.
\end{thm}
\begin{proof}
We must examine the initial conditions for $G_{\bn}^{(1)}$. We note that for any $\bn$ with some $n^{(\al)}=-1$, the function $G_\bn^{(1)}$
must vanish. Indeed, by definition it is the constant term in $\Q_{\al,1}$ of an expression with no non-negative power of $\Q_{\al,1}$ (as the insertion of $\Q_{\al,1}^{-1}$ cancels the prefactor $\Q_{\al,1}$, and the contributions 
from $\tau(\bz)$ only provide strictly negative powers of $\Q_{\al,1}$). 
We conclude that all values of $G_\bn^{(1)}=0$ for $\sigma(\bn)=-1,0$
except $G_{0,0,...,0}^{(1)}=1$ by the normalization of $\phi$. With these initial data, the $r$ difference equations determine a unique solution $G_\bn^{(1)}$ for all $\bn=(n^{(\al)})_{\al\in [1,r]}$ and $n^{(\al)}\geq 0$ for all $\al$.
\end{proof}

\begin{example}\label{twoexG}
When $r=2$ and $k=1$ (case of ${\mathfrak sl}_3$, level $1$), we have 
$ \Lambda_{1,1}=\Lambda_{2,2}=2$ and $ \Lambda_{1,2}=\Lambda_{2,1}=1$.
Denoting by $n=n^{(1)}_1$ and $p=n^{(2)}_1$, we have
the following recursion relation for $G_{n,p}\equiv G^{(1)}_{n,p}$:
\begin{equation}\label{levonethreeone}
v^3 G_{n+1,p}+(v^{-3n}-1)G_{n-1,p+1}+v^{-3-3n}(v^{-3p}-1)G_{n,p-1}= v^{-2n-p-1} e_1(\bz)\, G_{n,p}
\end{equation}
This equation does not determine $G_{n,p}$ entirely. We also have to consider the ``conjugate equation", obtained by insertion of the second conserved quantity $C_2$:
\begin{equation}\label{levonethreetwo}
v^3 G_{n,p+1}+(v^{-3p}-1)G_{n+1,p-1}+v^{-3-3p}(v^{-3n}-1)G_{n-1,p} =v^{-n-2p-1} e_2(\bz)\, G_{n,p}
\end{equation}
These two equations are readily seen to be three-term linear recursion relations in the variable $j=\sigma(n,p)=n+p$, namely allow to express a single function with $\sigma=j+1$ in terms of functions with $\sigma=j,j-1$.
Together with the initial data $G_{-1,p}=G_{n,-1}=0$ for all $n,p\geq 0$ and $G_{0,0}=1$ which determine all functions with $\sigma=-1,0$, the two above equations therefore determine $G_{n,p}$ completely. 
For instance, using the equations for all values of $\sigma=n+p$ indicated, we get:
\begin{eqnarray*}\sigma=0: \qquad  &&G_{1,0}=v^{-4}e_1\qquad G_{0,1}=v^{-4}e_2 \\
\sigma=1:\qquad  &&G_{2,0}=v^{-7}(v^{-3}e_1^2+(1-v^{-3})e_2)\qquad G_{1,1}=v^{-6}(v^{-3}e_1e_2+1-v^{-3})\\
&&G_{0,2}=v^{-7}(v^{-3}e_2^2+(1-v^{-3})e_1)
\end{eqnarray*}
with the shorthand $e_1=z_1+z_2+z_3$ and $e_2=z_1z_2+z_2z_3+z_1z_3$.
Note that the two equations determining $G_{1,1}$ are compatible, as a consequence of the 
commutation of $C_1$ and $C_2$ which implies $e_2 G_{1,0}=e_1 G_{0,1}$.
\end{example}

Theorem \ref{difconeN} may be immediately translated in terms of graded characters $\chi_\bn(q^{-1},\bz)$ by use of the formula \eqref{chitoG}, which results straightforwardly into the following:

\begin{thm}\label{difchar}
The graded characters $\chi_\bn\equiv \chi_\bn(q^{-1},\bz)$, 
$\bn=(n_i^{(\al)})_{\al\in [1,r];i\in [1,k]}$, $n_k^{(\al)},n_{k-1}^{(\al)}\geq 1-\delta_{k,1}$, 
satisfy the following difference equation for $k\geq 1$:
\begin{eqnarray}
&&\!\!\!\!\!\!\!\!\!\!\!\!\!\!\sum_{\al=1}^{r+1} 
\chi_{\bn+\epsilon_{\al-1,k-1}-\epsilon_{\al,k-1}+\epsilon_{\al,k}-\epsilon_{\al-1,k}} \nonumber \\
&&-\sum_{\al=1}^r q^{k-1-\sum_{i=1}^k i n_i^{(\al)}}
\chi_{\bn+\epsilon_{\al-1,k-1}-\epsilon_{\al,k-1}+\epsilon_{\al+1,k}-\epsilon_{\al,k}}= e_1(\bz) \, \chi_\bn
\label{diffeqchik}
\end{eqnarray}
with the convention that $(\epsilon_{\al,i})_j^{(\beta)}=\delta_{\beta,\al}\delta_{j,i}$, for $\beta\in [1,r]$ and $j\in [1,k]$,
$\epsilon_{0,i}=\epsilon_{r+1,i}=0$ for all $i$, and $e_1(\bz)=z_1+z_2+...+z_{r+1}$.
\end{thm}

The higher conserved quantities give rise to higher difference equations for $\chi_\bn$.

\begin{example}
In the case $r=1$ (${\mathfrak sl}_2$), we have:
$$\chi_{n_1,...,n_{k-1}-1,n_k+1}+\chi_{n_1,...,n_{k-1}+1,n_k-1}-q^{k-1-\sum_{i=1}^k in_i} 
\chi_{n_1,...,n_{k-1}-1,n_k-1}=(z+z^{-1})\, \chi_{n_1,...,n_{k-1},n_k}$$
For $k=1$, this reduces to:
\begin{equation}\label{chisltwo} \chi_{n+1}+(1-q^{-n})\chi_{n-1}=(z+z^{-1})\, \chi_{n}\end{equation}
\end{example}

\section{Difference equations for characters of level-1 (Weyl) modules}
\label{todawhit}
\subsection{Level-1 difference equations and the $q$-deformed open Toda chain}

When specialized to level $k=1$, the difference equation of Theorem \ref{difchar}
takes a particularly simple form. In this section, we show that this difference equation
is the eigenvalue equation for the q-deformed Toda operator for $U_q(\mathfrak{sl}_{r+1})$
of \cite{Etingof}, after applying a suitable automorphism, and performing a number of specializations.

We start with a few definitions.

\begin{defn}
We introduce the following difference operators acting on functions of the variable $\bx=(x_1,...,x_r)$:
\begin{eqnarray}
S_\al(f)(\bx)&=&f(x_1,...,x_{\al-1},x_\al-\frac{1}{2},x_{\al+1},...,x_r)\qquad (\al=1,2,...,r) \\
S_0(f)(\bx)&=&f(\bx), \qquad  S_{r+1}(f)(\bx)=f(\bx)\\
T_\al&=& S_{\al+1}S_{\al}^{-1} \qquad (\al=0,1,...,r)
\end{eqnarray}
\end{defn}

\begin{defn}
The $\tilde q$-deformed difference (open) Toda Hamiltonian \cite{Etingof} for $U_{\tilde q}(\mathfrak{sl}_{r+1})$
is the following operator acting on functions of $\bx$, for fixed parameters ${\tilde q}, \nu_\al \in \C^*$:
\begin{equation}\label{qtoda}
H_{\tilde q}=\sum_{\al=0}^r T_\al^2+({\tilde q}-{\tilde q}^{-1})^2 \sum_{\al=1}^{r} \nu_\al \, {\tilde q}^{-2x_\al} T_{\al-1} T_{\al} 
\end{equation}
\end{defn}

In  \cite{Etingof}, this Hamiltonian is related to the so-called relativistic Toda operator
by use of an automorphism.

\begin{defn} We introduce the following automorphism $\tau$ of the algebra ${\mathcal T}$ generated by $T_\al$,
$\al=0,1,...,r$ and $U_\al={\tilde q}^{-2x_\al}$, $\al=1,2,...,r$:
\begin{equation}\label{isom2}
\tau(T_\al)=T_\al \qquad \tau(U_\al)=U_\al\, T_\al \,T_{\al-1}^{-1}
\end{equation}
In particular, the automorphism $\tau$ respects the commutation relations 
$T_\al \, U_\beta=q^{\delta_{\beta,\al+1}-\delta_{\beta,\al}}U_\beta\, T_\al$.
\end{defn}

The image of $H_q$ under $\tau$ is the following Hamiltonian:
\begin{equation}\label{todaq} H_{\tilde q}'=\tau(H_{\tilde q})=T_0^2+\sum_{i=1}^r 
\left(1+({\tilde q}-{\tilde q}^{-1})^2 \nu_\al\, {\tilde q}^{-2x_\al} \right)T_{\al}^2
\end{equation}

On the other hand, it is easy to rewrite the level-1 difference equation \eqref{diffeqchik} of Theorem \ref{difchar} for 
$\chi_\bn\equiv \chi_\bn(q^{-1},\bz)$, $n=\{n^{(\al)}\}_{\al\in [1,r]}$ as:
\begin{equation}\label{gradedtoda} 
\left(\sum_{\al=0}^{r} T_{\al}^{-2} -\sum_{\al=1}^{r} q^{-n^{(\al)}} T_{\al}^{-2} \right)\, \chi_\bn=e_1(\bz) \chi_\bn 
\end{equation}
where the operator $T_\al^2$ acts on functions of $\bn$ by 
$(T_\al^2 f)(\bn)=f(\bn+\epsilon_{\al}-\epsilon_{\al+1})$.

We see that if we pick
\begin{equation}\label{specia} \bx=-\bn,\quad {\tilde q}^2=q^{-1},\quad ({\tilde q}-{\tilde q}^{-1})^2 \nu_\al=-1 
\end{equation}
then the difference equation turns into an eigenvector equation for the $\tilde q$-Toda Hamiltonian \eqref{todaq},
with eigenvalue $e_1(\bz)$. 
The level-1 graded character is therefore a $\tilde q$-Whittaker function.

\subsection{Whittaker functions and fusion products}

In \cite{DKT}, we have obtained the so-called fundamental q-Whittaker functions $W_\lambda(\bx)$ for 
$U_q(\mathfrak{sl}_{r+1})$ by explicitly constructing Whittaker vectors in a Verma module $V_\lambda$
with generic highest weight $\lambda$, using a path model. These form a basis of the eigenspace of the q-Toda Hamiltonian \eqref{qtoda} for eigenvalue $E_\lambda=\sum_{i=0}^r q^{2(\lambda+\rho\vert \omega_{i+1}-\omega_i)}$
where $\omega_i$ are the fundamental weights of $A_r$. 
The dimension of this eigenspace is the order of the Weyl group, here $(r+1)!$, as we may generate 
other independent solutions $W_{s(\lambda+\rho)-\rho}(\bx)$ by Weyl group reflections $s$, while preserving
$E_{s(\lambda+\rho)-\rho}=E_\lambda$.

Identifying $z_i={\tilde q}^{2(\lambda+\rho\vert \omega_{i}-\omega_{i-1})}$ for $i=1,2,...,r+1$, we deduce that
the graded level-1 character $\chi_\bn$ is a linear combination of the image of the fundamental 
${\tilde q}$-Whittaker functions 
under the automorphism $\tau$, with the additional specialization \eqref{specia}.

Let us illustrate this in the case of $U_q(\mathfrak{sl}_2)$.
The ${\tilde q}$-Toda eigenvector equation is:
$$ W_\lambda(x-1)+(1+({\tilde q}-{\tilde q}^{-1})^2 \nu {\tilde q}^{-2x})W_\lambda(x+1)=(p
+p^{-1})W(x),\qquad  p={\tilde q}^{-\frac{\lambda+1}{2}} $$
Applying the automorphism $\tau$ and using the specialization \eqref{specia}, we obtain
a transformed fundamental ${\tilde q}$-Whittaker function $W'_\lambda(n)$, with the following series expansion
(valid for $|q|>1$):
$$W'_\lambda(n)=\tau(W_\lambda)(-n)=p^{n-\frac{1}{2}} \sum_{a\in \Z_+} 
\frac{{q}^{-a (n+1)}}{\prod_{i=1}^a (1-{q}^{-i})(1-p^{2}{q}^{-i})} $$
Analogously, we have
the Weyl-reflected fundamental q-Whittaker function:
$$W'_{-\lambda-2}(n)=p^{\frac{1}{2}-n} \sum_{a\in \Z_+} 
\frac{{q}^{-a (n+1)}}{\prod_{i=1}^a (1-{q}^{-i})(1-p^{-2}{q}^{-i})} $$
The functions $W'_\lambda(n),W'_{-\lambda-2}(n)$ form a basis of the eigenspace of the transformed
${\tilde q}$-Toda Hamiltonian with same eigenvalue, namely
\begin{equation}\label{witoad}\tau(H_{\tilde q}) W'(n)=W'(n+1)+(1-q^{-n})W'(n-1)=(p+p^{-1})W'(n) \end{equation}
This coincides with the level-1 difference equation \eqref{chisltwo} with $p=z$.
Looking for a linear combination $\chi_n=c_\lambda(p,q) \, W'_\lambda(n)+c_{-\lambda-2}(p,q)\, W'_{-\lambda-2}(n)$ 
for say $n=0,1$, we find the coefficients:
$$c_\lambda(p,q)=\frac{ p^{\frac{1}{2}}}{(1-p^{-2})\prod_{i=1}^\infty (1-p^{-2}{q}^{-i})},
\quad c_{-\lambda-2}(p,q)=c_\lambda(p^{-1},q)=\frac{p^{-\frac{1}{2}}}{(1-p^{2})\prod_{i=1}^\infty (1-p^{2}{q}^{-i})}  $$
Remarkably, we have realized the graded character, which is polynomial in $q^{-1}$, $p,p^{-1}$
as a linear combination of two infinite series of $q^{-1}$ (the fundamental q-Whittaker functions). 
The cancellations occurring are the q-deformed version of the so-called class 1 regularity condition on
Whittaker functions. So we may view the graded character as a class one specialized q-Whittaker function.
We expect this to generalize to $U_q(\mathfrak{sl}_{r+1})$ (see also \cite{Lebedev,Lebedev3} for analogous 
considerations).

\section{The solution for $\sl_{r+1}$}
\label{demazure}
In this section, we introduce a generalization of the specialized Macdonald difference operators
(corresponding to their ``dual Whittaker limit" $t\to \infty$), and use them to construct a solution of the difference equations for the $\sl_{r+1}$ graded characters, by iterated action on the constant function $1$. 
We shall proceed in several steps. After introducing the new difference operators, we show that they satisfy the dual quantum $Q$-system. This allows to consider them as raising operators for graded characters, Theorems \ref{raisingthm} and \ref{almost}, which are proved in two separate steps, first only for level $k=1$ and then for general level $k\geq 2$.

\subsection{A realization of the dual quantum $Q$-system via generalized Macdonald operators}

\begin{defn}\label{hellrise}
Recall the notations \eqref{productnotation} $z_I,D_I,a_I(\bz)$.
We have the following sequence of  operators ${\mathcal D}_{\al,n}$, $\al=0,1,...,r+1$ and $n\in \Z$:
\begin{equation}\label{qan}
{\mathcal D}_{\al,n}=v^{-\frac{\Lambda_{\al,\al}}{2}n-\sum_{\beta=1}^r \Lambda_{\al,\beta}}
\sum_{I\subset [1,r+1]\atop |I|=\al } (z_I)^n a_I(\bz)\, D_I
\end{equation}
In particular we have 
$$ {\mathcal D}_{0,n}=1 \quad {\rm and}\quad {\mathcal D}_{r+1,n}
=(z_1z_2\cdots z_{r+1})^n D_1D_2\cdots D_{r+1}=(z_1z_2\cdots z_{r+1})^n=1$$
\end{defn}

Recall the standard definition of the difference Macdonald operators
for ${\mathfrak sl}_{r+1}$ \cite{macdo}:
\begin{equation}\label{macdop}
M_\al^{q,t}=\sum_{I\subset [1,r+1]\atop |I|=\al } \prod_{i\in I\atop j\not\in I} \frac{t z_i-z_j}{z_i-z_j} \Gamma_I
\end{equation}
where $\Gamma_I=\prod_{i\in I}\Gamma_i$, and $\Gamma_i(\bz)=(z_1,...,z_{i-1},q\, z_i,z_{i+1},...,z_{r+1})$.
This expression allows to identify  our operators ${\mathcal D}_{\al,n}$ for $n=0$ as:
\begin{eqnarray}
{\mathcal D}_{\al,0}&=&v^{-\sum_{\beta=1}^r \Lambda_{\al,\beta}} \, M_\al\, \Delta^\al\nonumber \\
M_\al &:=&\lim_{t\to \infty} t^{-\al(r+1-\al)}M_\al^{q,t}= \sum_{I\subset [1,r+1]\atop |I|=\al } a_I(\bz)\,  \Gamma_I
\label{degmac}
\end{eqnarray}
where $\Delta(\bz)=(v z_1,...,v z_{r+1})$.
The operators $M_\al^{q,t}$ as well as their limits $M_\al$, $\al=0,1,...,r+1$ are known to form a commuting family.

\begin{remark}
In the theory of Macdonald polynomials and difference operators, 
the limit $t\to \infty$ may be thought of as a ``dual Whittaker limit". Indeed, as pointed out below,
the duality of Macdonald polynomials $P_\lambda^{q^{-1},t^{-1}}=P_\lambda^{q,t}$,  allows to relate
our limit $t\to\infty$ to the so-called q-Whittaker limit $t\to 0$.
\end{remark}

Let $\mathcal A^{*}$ be the algebra generated by $\{Q_{\al,k}^*: \al\in[1,r],k\in\Z\}$ over $\Z_v$ modulo the ideal generated by the relations
\begin{eqnarray} \label{dualqQsys}
\Q_{\al,n}^*\, \Q_{\beta, p}^*&=&
v^{-\Lambda_{\al,\beta}(p-n)} \, \Q_{\beta,p}^*\, \Q_{\al,n}^* \qquad (|p-n|\leq |\beta-\al|+1)\\
v^{\Lambda_{\al,\al}}\, \Q_{\al,n-1}^* \, \Q_{\al,n+1}^* &=& 
(\Q_{\al,n}^*)^2-\Q_{\al+1,n}^*\,\Q_{\al-1,n}^*\nonumber
\end{eqnarray}
Equivalently, the second relation may be rewritten, using \eqref{dualqQsys} as:
\begin{equation}\label{leftq}
v^{-\Lambda_{\al,\al}}\, \Q_{\al,n+1}^*\, \Q_{\al,n-1}^*=(\Q_{\al,n}^*)^2-v^{-r-1}\Q_{\al+1,n}^*\,\Q_{\al-1,n}^*
\end{equation}
We refer to this as the dual quantum $Q$-system.  The algebra $\mathcal A^*$ is isomorphic to the algebra 
$\mathcal A^{\rm op}$, with the opposite multiplication to $\mathcal A$.

We have the following main result.

\begin{thm}\label{repsQ}
We have a polynomial representation $\pi$
of ${\mathcal A}^*$, with $\pi(\Q_{\al,n}^*)={\mathcal D}_{\al,n}$ of \eqref{qan}. 
That is, acting by left multiplication on the space $\C[\bz]$, the operators ${\mathcal D}_{\al,n}$  obey
the dual quantum $Q$-system relations for $A_r$:
\begin{eqnarray}
{\mathcal D}_{\al,n}\, {\mathcal D}_{\beta,p} &=&v^{-\Lambda_{\al,\beta}(p-n)} \, {\mathcal D}_{\beta,p} \, {\mathcal D}_{\al,n} 
\qquad (|p-n|\leq |\beta-\al|+1)\label{commutdem}\\
v^{-\Lambda_{\al,\al}}\, {\mathcal D}_{\al,n+1}\, {\mathcal D}_{\al,n-1}&=&({\mathcal D}_{\al,n})^2
-v^{-r-1}{\mathcal D}_{\al+1,n}\,{\mathcal D}_{\al-1,n}
\label{qsysdem}
\end{eqnarray}
\end{thm}

Note that when $n=p=0$ the relation \eqref{commutdem} boils down to the commutation of the specialized
Macdonald operators at $t\to\infty$, as $M_\al\,\Delta=\Delta\,M_\al$.

The remainder of this section is devoted to the proof of this theorem.

Let us define for any disjoint sets $I$, $J$ of indices the quantities:
\begin{equation}
a_{I,J}(\bz)=\prod_{i\in I\atop j\in J} \frac{z_i}{z_i-z_j},\ \ 
b_{I,J}(\bz)=\prod_{i\in I\atop j\in J} \frac{z_i}{z_i-q z_j}, \ \ 
c_{I,J}(\bz)=\prod_{i\in I\atop j\in J} \frac{q z_i}{q z_i-z_j}
\end{equation}
Note that in this notation $a_I(\bz)$ of eq.\eqref{productnotation}  is simply $a_{I,\bar I}(\bz)$.

We have the following two lemmas. 

\begin{lemma}\label{firstlemma}
Fix integers $0\leq a\leq b$ and $\bz=(z_1,z_2,...,z_{a+b})$. Then we have:
\begin{equation}\label{lemone}
\sum_{I\cup J=[1,a+b],\ I\cap J=\emptyset\atop
|I|=a,\ |J|=b} (z_J)^p \left( a_{I,J}(\bz) b_{J,I}(\bz) -q^{p a} a_{J,I}(\bz) b_{I,J}(\bz)\right) =0\quad (|p|\leq b-a+1)
\end{equation}
\end{lemma}

\begin{lemma}\label{secondlemma}
Fix an integer $a\geq 1$ and $\bz=(z_1,z_2,...,z_{2a})$. Then we have:
\begin{equation}\label{lemtwo}
\sum_{I\cup J=[1,2a],\ I\cap J=\emptyset\atop
|I|=|J|=a} a_{I,J}(\bz) b_{J,I}(\bz)\left( 1-q^a \frac{z_{I}}{z_J}\right) =
\sum_{I\cup J=[1,2a],\ I\cap J=\emptyset\atop
|I|=a+1,\ |J|=a-1}a_{I,J}(\bz) b_{J,I}(\bz) 
\end{equation}
\end{lemma}

The above Lemmas \ref{firstlemma} and \ref{secondlemma} are proved in Appendix A below.
Let us now turn to the proof of Theorem \ref{repsQ}.  
Let us first compute the quantity ${\mathcal D}_{\al,n}\, {\mathcal D}_{\beta,p}$.
Substituting the definition \eqref{qan}, we get:
\begin{eqnarray*}
&&{\mathcal D}_{\al,n}\, {\mathcal D}_{\beta,p}=\sum_{I,J\subset [1,r+1]\atop |I|=\al,\ |J|=\beta}
(z_I)^n\,  a_I(\bz)\,  D_I \, (z_J)^p \, a_J(\bz) \, D_J\\
&=& \sum_{K\subset L\subset [1,r+1]\atop |L|\leq \al+\beta,\ |K|\leq \al,\beta} 
\sum_{I_0\cup J_0=L\setminus K, I_0\cap J_0=\emptyset\atop
(I=K\cup I_0,\ J=K\cup J_0)} (z_{I_0}z_K)^n(z_{J_0}z_K)^p \,  a_{K\cup I_0,\bar L \cup J_0}\, v^{\al\beta p}
D_K D_{I_0} a_{K\cup  J_0,\bar L \cup I_0} D_{J}\\
&=& v^{\al\beta p}\!\!\!\!\! \!\!\!\!\! \sum_{K\subset L\subset [1,r+1]\atop |L|\leq \al+\beta,\ |K|\leq \al,\beta} \!\!\!\!\! (z_{L}z_K)^n
\!\!\!\!\! \!\!\!\!\! \sum_{I_0\cup J_0=L\setminus K,  I_0\cap J_0=\emptyset\atop |I_0|=\al-|K|,\ |J_0|=\beta-|K|}
\!\!\!\!\! (z_{J_0})^{p-n} 
a_{K,\bar L} a_{I_0,\bar L} a_{K,J_0} a_{I_0,J_0}
c_{K,\bar L} a_{K,I_0} a_{J_0,\bar L} b_{J_0,I_0} D_ID_J\\
&=&v^{\al\beta p}\!\!\!\!\! \!\!\!\!\! \sum_{K\subset L\subset [1,r+1]\atop |L|\leq \al+\beta,\ |K|\leq \al,\beta}
\!\!\!\!\!  (z_{L}z_K)^n
a_{K,\bar L}c_{K,\bar L}a_{K,L\setminus K}a_{L\setminus K,\bar L}
\left(\sum_{I_0\cup J_0=L\setminus K,  I_0\cap J_0=\emptyset\atop |I_0|=\al-|K|,\ |J_0|=\beta-|K|} 
(z_{J_0})^{p-n} a_{I_0,J_0}b_{J_0,I_0}\right)D_KD_L
\end{eqnarray*}
where we have replaced the sum over $I,J$ by one over $K=I\cap J$ and $L=I\cup J$ first, and then written
the disjoint unions
$I=K\cup I_0$, $J=K\cup J_0$, $\bar I=\bar L \cup J_0$, and $\bar J=\bar L\cup I_0$. 
Note that we have isolated a factor
$u_{K,L}(n):= (z_{L}z_K)^n
a_{K,\bar L}c_{K,\bar L}a_{K,L\setminus K}a_{L\setminus K,\bar L}$ which does not depend on $I_0,J_0$.
We may now write:
\begin{eqnarray*}
&&
v^{\frac{n\Lambda_{\al,\al}+p\Lambda_{\beta,\beta}}{2}+\sum_{\gamma}\Lambda_{\al,\gamma}+\Lambda_{\beta,\gamma}}
\left\{{\mathcal D}_{\al,n}\, {\mathcal D}_{\beta,p}-v^{-\Lambda_{\al,\beta}(p-n)}{\mathcal D}_{\beta,p}\,{\mathcal D}_{\al,n}\right\}\\
&=&\!\!\!\!\! \!\!\!\!\! \sum_{K\subset L\subset [1,r+1]\atop |L|\leq \al+\beta,\ |K|\leq \al,\beta}\!\!\!\!\!  u_{K,L}(n)
\sum_{I_0\cup J_0=L\setminus K,  I_0\cap J_0=\emptyset\atop |I_0|=\al-|K|,\ |J_0|=\beta-|K|} 
(z_{J_0})^{p-n} \Big(v^{\al\beta p} a_{I_0,J_0}b_{J_0,I_0}-v^{\al\beta n-\Lambda_{\al,\beta}(p-n)}
a_{J_0,I_0}b_{I_0,J_0}\Big)D_KD_L\\
&=&v^{\al\beta p} \!\!\!\!\! \!\!\!\!\! \sum_{K\subset L\subset [1,r+1]\atop |L|\leq \al+\beta,\ |K|\leq \al,\beta}\!\!\!\!\!  u_{K,L}(n)
\sum_{I_0\cup J_0=L\setminus K,  I_0\cap J_0=\emptyset\atop |I_0|=\al-|K|,\ |J_0|=\beta-|K|} 
(z_{J_0})^{p-n} \Big(a_{I_0,J_0}b_{J_0,I_0}-q^{\al (p-n)}
a_{J_0,I_0}b_{I_0,J_0}\Big)D_KD_L=0
\end{eqnarray*}
where we have first used $\Lambda_{\al,\beta}+\al\beta=\al(r+1)$ for $\al\leq \beta$,  $q=v^{-(r+1)}$, and then
applied Lemma \ref{firstlemma} for every fixed pair $K,L$ to the second summation, with $a=\al-|K|\leq b=\beta-|K|$
and $|p-n|\leq b-a+1=\beta-\al+1$.
The relation \eqref{commutdem} follows.

Analogously, we compute:
\begin{eqnarray}
&&v^{n\Lambda_{\al,\al}+2\sum_\beta \Lambda_{\al,\beta}}\left\{
({\mathcal D}_{\al,n})^2-v^{-\Lambda_{\al,\al}}\,{\mathcal D}_{\al,n+1}\,{\mathcal D}_{\al,n-1}\right\}\nonumber \\
&&\qquad =\sum_{K\subset L\subset [1,r+1]\atop |L|\leq 2\al,\ |K|\leq \al }\!\!\!\!\!  u_{K,L}(n)
\sum_{I_0\cup J_0=L\setminus K,  I_0\cap J_0=\emptyset\atop |I_0|=|J_0|=\al-|K|} 
a_{I_0,J_0}b_{J_0,I_0} \left( v^{n\al^2}-v^{(n-1)\al^2-\Lambda_{\al,\al}} \frac{z_{I_0}}{z_{J_0}}\right)D_KD_L\nonumber \\
&&\qquad =v^{n\al^2}\sum_{K\subset L\subset [1,r+1]\atop |L|\leq 2\al,\ |K|\leq \al }\!\!\!\!\!  u_{K,L}(n)
\sum_{I_0\cup J_0=L\setminus K,  I_0\cap J_0=\emptyset\atop |I_0|=|J_0|=\al-|K|} 
a_{I_0,J_0}b_{J_0,I_0} \left(1-q^{\al}\,  \frac{z_{I_0}}{z_{J_0}}\right)D_KD_L\label{onelem}
\end{eqnarray}
and finally 
\begin{eqnarray}
&&v^{n\Lambda_{\al,\al}+2\sum_\beta \Lambda_{\al,\beta}-r-1}{\mathcal D}_{\al+1,n}\, {\mathcal D}_{\al-1,n}
=v^nv^{\frac{1}{2}(\Lambda_{\al+1,\al+1}+\Lambda_{\al-1,\al-1})n+\sum_\beta \Lambda_{\al+1,\beta}+\Lambda_{\al-1,\beta}}
{\mathcal D}_{\al+1,n}\, {\mathcal D}_{\al-1,n} \nonumber \\
&&\qquad\qquad =v^{n\al^2} \sum_{K\subset L\subset [1,r+1]\atop |L|\leq 2\al,\ |K|\leq \al}\!\!\!\!\!  u_{K,L}(n)
\sum_{I_0\cup J_0=L\setminus K,  I_0\cap J_0=\emptyset\atop |I_0|=\al+1-|K|, |J_0|=\al-1-|K|} 
a_{I_0,J_0}b_{J_0,I_0} D_KD_L\label{twolem}
\end{eqnarray}
where we have used the relations
\begin{eqnarray*}
2+ \Lambda_{\al+1,\al+1}+\Lambda_{\al-1,\al-1}-2\Lambda_{\al,\al}&=&0\\
\Lambda_{\al+1,\beta}+\Lambda_{\al-1,\beta}-2\Lambda_{\al,\beta}&=&-(r+1)\delta_{\al,\beta} 
\end{eqnarray*}
The relation \eqref{qsysdem} follows by identifying equations \eqref{onelem} and \eqref{twolem}
by applying Lemma \ref{secondlemma} for $a=\al-|K|$ to the second summation for $K,L$ fixed.
This completes the proof of Theorem \ref{repsQ}.

\subsection{Graded characters and difference raising operators}

\subsubsection{The main results}

In this section, we show that, in a way analogous to how the Kirillov-Noumi difference operators are raising operators
for Macdonald polynomials \cite{kinoum},  
our generalized degenerate Macdonald operators are raising operators for the graded characters.

\begin{thm}\label{raisingthm}
For $\bn=\{ n_i^{(\al)}\}_{\al\in [1,r];i\in \Z_{>0}}$, the coefficients $G^{(k)}_\bn$ \eqref{defGkn} for $A_r$ at level $k$ are given by
the iterated action of the generalized Macdonald operators \eqref{qan} on the constant function $1$:
\begin{equation}\label{Gfound} 
G^{(k)}_\bn =\prod_{\al=1}^r({\mathcal D}_{\al,k})^{n_k^{(\al)}}\prod_{\al=1}^r({\mathcal D}_{\al,k-1})^{n_{k-1}^{(\al)}}\cdots
\prod_{\al=1}^r({\mathcal D}_{\al,1})^{n_1^{(\al)}}\, 1 
\end{equation}
\end{thm}

Using the relation \eqref{chitoG}, we immediately deduce the following:

\begin{thm}\label{almost}
The graded characters for ${\mathfrak sl}_{r+1}$ at level $k$ are given by:
\begin{eqnarray}
\chi_\bn(q^{-1},\bz)&= &v^{\frac{1}{2} \sum_{i,j,\al,\beta} n_i^{(\al)}{\rm Min}(i,j)\Lambda_{\al,\beta}n_j^{(\beta)}+\sum_{i,\al,\beta} n_i^{(\al)}\Lambda_{\al,\beta}+\frac{1}{2}\sum_\al \Lambda_{\al,\al}+\sum_{\al<\beta}\Lambda_{\al,\beta}} \nonumber \\
&& \quad \times  \prod_{\al=1}^r({\mathcal D}_{\al,k})^{n_k^{(\al)}}\prod_{\al=1}^r({\mathcal D}_{\al,k-1})^{n_{k-1}^{(\al)}}\cdots
\prod_{\al=1}^r({\mathcal D}_{\al,1})^{n_1^{(\al)}}\, 1
\end{eqnarray}
\end{thm}

Relaxing the condition $z_1z_2\cdots z_{r+1}=1$, we may restate this result in terms of the family of difference operators 
$M_{\al,n}$ defined as:
\begin{equation}\label{otmacdo}
M_{\al,n}=\sum_{I\subset [1,r+1]\atop |I|=\al } (z_I)^n a_I(\bz)\, \Gamma_I
=v^{\frac{\Lambda_{\al,\al}}{2}n+\sum_\beta \Lambda_{\al,\beta}}\, {\mathcal D}_{\al,n} \Delta^{-\al}
\end{equation}
These satisfy a renormalized version of the dual quantum $Q$-system:
\begin{eqnarray}
M_{\al,n}\, M_{\beta,p} &=&q^{{\rm Min}(\al,\beta)(p-n)} \, M_{\beta,p} \, M_{\al,n} 
\qquad (|p-n|\leq |\beta-\al|+1)\label{commutdemnew}\\
q^{\al}\, M_{\al,n+1}\, M_{\al,n-1}&=&(M_{\al,n})^2-M_{\al+1,n}\,M_{\al-1,n}\qquad (\al\in [1,r];n\in \Z)
\label{qsysdemnew}
\end{eqnarray}
with $M_{0,n}=1$ and $M_{r+1,n}=(z_1z_2\cdots z_{r+1})^n \Delta^{-r-1}$. Note also that $M_{\al,0}$ is equal to the 
degenerate Macdonald operator $M_\al$ 
of eq.\eqref{degmac}.
We have:
\begin{cor}\label{gracor}
The graded characters for ${\mathfrak sl}_{r+1}$ at level $k$ are given by:
\begin{eqnarray}
\chi_\bn(q^{-1},\bz)&= &q^{-\frac{1}{2} \sum_{i,j,\al,\beta} n_i^{(\al)}{\rm Min}(i,j){\rm Min}(\al,\beta)n_j^{(\beta)}+\frac{1}{2}\sum_{i,\al} i\al n_i^{(\al)}} \nonumber \\
&& \quad \times  \prod_{\al=1}^r(M_{\al,k})^{n_k^{(\al)}}\prod_{\al=1}^r(M_{\al,k-1})^{n_{k-1}^{(\al)}}\cdots
\prod_{\al=1}^r(M_{\al,1})^{n_1^{(\al)}}\, 1
\end{eqnarray}
\end{cor}
\begin{proof}
We use the relation \eqref{otmacdo} to rewrite the result of Theorem \ref{almost}. We make use of the commutation
relation $\Delta M_{\al,n}=v^{n\al} M_{\al,n}\Delta$, and of $\Lambda_{\al,\beta}+\al \beta=(r+1){\rm Min}(\al,\beta)$.
\end{proof}

\begin{remark}
The iterated action of the raising operators $M_{\al,n}$ on the function $1$ results clearly in a symmetric polynomial of the $z$'s 
with coefficients that are polynomial in $q$. 
On the other hand, the prefactor is a negative integer power of $q$, as
\begin{eqnarray*} &&\frac{1}{2} \sum_{i,j,\al,\beta} n_i^{(\al)}{\rm Min}(i,j){\rm Min}(\al,\beta)n_j^{(\beta)}-\frac{1}{2}\sum_{i,\al} i\al n_i^{(\al)}\\
&&\qquad =
\sum_{i,\al}i\al \frac{n_i^{(\al)}(n_i^{(\al)}-1)}{2}+\sum_{i<j\, {\rm or}\, \al<\beta}n_i^{(\al)}{\rm Min}(i,j){\rm Min}(\al,\beta)n_j^{(\beta)}\in\Z_+
\end{eqnarray*}
We deduce that $\chi_\bn(q^{-1},\bz)$ is a polynomial of $q,q^{-1}$. Moreover the graded characters have the limit
$\lim_{q\to\infty} \chi_\bn(q^{-1},\bz)=s_\lambda(\bz)$ which sends the graded tensor product to its top component,
with $\lambda_\al=\sum_{i=1}^k i n_i^{(\al)}$, 
hence $\chi_\bn(q^{-1},\bz)$ is a polynomial of $q^{-1}$,
as expected from its definition.
\end{remark}

\subsubsection{Proof in the case of level $1$}\label{leveoneproof}

Let us now turn to the proof of Theorem \ref{raisingthm}. We will proceed in two steps. First, we will show the theorem in the case $k=1$ only.
The idea is to show that the expression \eqref{Gfound}
satisfies all the difference equations that determine $G^{(1)}_\bn$. 
To this end we use the conserved quantities of the dual quantum $Q$-system, 
easily obtained by applying the anti-homorphism 
$$*: {\mathcal A}\to {\mathcal A}^*, \qquad  \Q_{\al,k}\mapsto \Q_{\al,k}^*$$
such that $(AB)^*=B^*A^*$ for all $A,B\in {\mathcal A}$ and $v^*=v$,
and then evaluating in the polynomial representation where $\pi(\Q_{\al,k}^*)={\mathcal D}_{\al,k}$.
The quantities ${\mathcal Y}_\al(n)=\pi(y_\al(n)^*)$
of (\ref{yodd}-\ref{yeven}) are expressed in terms of  ${\mathcal D}_{\al,n}$, ${\mathcal D}_{\al,n+1}$ as:
\begin{eqnarray*}
{\mathcal Y}_{2\al-1}(n)&=&
{\mathcal D}_{\al-1,n}{\mathcal D}_{\al,n}^{-1}{\mathcal D}_{\al-1,n+1}^{-1}{\mathcal D}_{\al,n+1}\\
{\mathcal Y}_{2\al}(n)&=&
-{\mathcal D}_{\al-1,n}{\mathcal D}_{\al,n}^{-1}{\mathcal D}_{\al,n+1}^{-1}{\mathcal D}_{\al+1,n+1}
\end{eqnarray*}
Here, we use the formal (left and right) inverse ${\mathcal D}_{\al,k}^{-1}$ of the difference operator 
${\mathcal D}_{\al,k}$ defined as follows. If $|v|>1$, setting $I_\al=\{1,2,...,\al\}$, we write the convergent series:
\begin{eqnarray*}
{\mathcal D}_{\al,k}^{-1}&=&\left(z_{I_\al}^k a_{I_\al}(\bz) D_{I_\al}\left( \sum_{I\subset [1,r+1],\, |I|=\al}
D_{I_\al}^{-1} \frac{z_I^k a_I(\bz)}{z_{I_\al}^k a_{I_\al}(\bz)}  D_I \right)\right)^{-1} \\
&=& \sum_{n\geq 0} \left(\sum_{I\subset [1,r+1],\, |I|=\al}
D_{I_\al}^{-1} \frac{z_I^k a_I(\bz)}{z_{I_\al}^k a_{I_\al}(\bz)}  D_I \right)^n\, 
D_{I_\al}^{-1}\frac{1}{z_{I_\al}^k a_{I_\al}(\bz) }
\end{eqnarray*}
If $|v|<1$, we must use ${\bar I_{r-\al+1}}=\{r+1,r,...,r-\al+2\}$ instead of $I_\al$.

Noting that $*$ is an anti-homorphism which inverts the order of weights, we get the following:
\begin{lemma}
The conserved quantities ${\mathcal C}_m=\pi(C_m^*)$, $m=0,1,...,r+1$ of the dual quantum 
$Q$-system are expressed in terms of 
the operators ${\mathcal D}_{\al,n}$, ${\mathcal D}_{\al,n+1}$ as:
\begin{equation}\label{consdem}
{\mathcal C}_m =\sum_{{\rm Hard}\, {\rm Particle}\,  {\rm configurations}\atop
i_1<i_2<\cdot <i_m\,  {\rm on}\, {\mathcal G}_r} {\mathcal Y}_{i_1}(n){\mathcal Y}_{i_2}(n)\cdots {\mathcal Y}_{i_m}(n) 
\end{equation}
\end{lemma}

For instance, we have the first non-trivial conserved quantity, obtained from \eqref{finacone}:
\begin{equation}\label{lefinacone}
{\mathcal C}_1=\sum_{\al=1}^{r+1} v^{r} 
{\mathcal D}_{\al-1,n+1}^{-1}{\mathcal D}_{\al,n+1}{\mathcal D}_{\al-1,n}{\mathcal D}_{\al,n}^{-1}
-\sum_{\al=1}^r 
v^{-1}{\mathcal D}_{\al,n+1}^{-1}{\mathcal D}_{\al+1,n+1} {\mathcal D}_{\al-1,n}{\mathcal D}_{\al,n}^{-1}
\end{equation}

All quantities ${\mathcal C}_m$ \eqref{consdem} are conserved i.e. they are independent of $n$,
and we may in particular express them in the limit $n\to\infty$ as we did before. 

\begin{thm}\label{actdem}
For all $m=0,1,...,r+1$,
the conserved quantity ${\mathcal C}_m$ \eqref{consdem} of the dual quantum $Q$-system acts on functions of 
$\bz$ by multiplication by $v^{\frac{mr}{2}}$ times 
the $m$-th elementary symmetric function $e_m(\bz)$,
namely
$${\mathcal C}_m=v^{\frac{mr}{2}} \sum_{1\leq i_1<i_2<\cdots <i_m \leq r+1} z_{i_1}z_{i_2}\cdots z_{i_m} 
=v^{\frac{mr}{2}} \, e_m(\bz)$$
\end{thm}
\begin{proof}
We will compute the action of ${\mathcal C}_m$ expressed as \eqref{consdem} in the limit when $n\to \infty$.
We must estimate the operator
${\mathcal D}_{\al,n}$ when $n$ becomes large. To this end, and without loss of generality, 
let us assume the modules of the $z_i$'s are strictly ordered,
say $|z_1|>|z_2|>\cdots >|z_{r+1}|>0$. 
Then for large $n$ the expression for ${\mathcal D}_{\al,n}$ is dominated
by the contribution of the subset $I_\al=\{1,2,...,\al\}$, and we have
$${\mathcal D}_{\al,n}\sim v^{-\frac{\Lambda_{\al,\al}}{2}n-\sum_\beta \Lambda_{\al,\beta}}
(z_{I_\al})^n a_{I_\al}(\bz) D_{I_\al}\quad {\rm hence}\quad  
\lim_{n\to \infty} {\mathcal D}_{\al,n+1}^{-1}{\mathcal D}_{\al,n}= v^{-\frac{\Lambda_{\al,\al}}{2}} z_{I_\al}^{-1}$$
This gives
\begin{eqnarray*}
\lim_{n\to \infty} {\mathcal Y}_{2\al-1}(n)&=&v^{\frac{r}{2}}z_\al \qquad (\al=1,2,...,r+1)\\
\lim_{n\to \infty} {\mathcal Y}_{2\al}(n)&=&0\qquad (\al=1,2,...,r)
\end{eqnarray*}
as the latter is proportional to $(z_{\al+1}/z_\al)^n\to 0$ when $n\to \infty$. 
As before, the hard particle model reduces to that on the odd vertices of 
${\mathcal G}_r$ which are not connected by edges, hence the partition functions are simply the elementary symmetric functions
of the variables $v^{\frac{r}{2}}z_\al$, $\al=1,2...,r+1$ and the theorem follows.
\end{proof}

We are now ready to prove Theorem \ref{raisingthm} in the case of level $k=1$. 
We will show that the function \eqref{Gfound} for $k=1$ satisfies the {\it same} difference equation
\eqref{diffeqkone}
as in Theorem \ref{difconeN} and its higher $m$ versions. 

First,
we may identify the action of the conserved quantity $C_m$ on the function $\tau(\bz)$ 
within the constant term evaluation of Corollary \ref{conscor}
with that of the conserved quantity ${\mathcal C}_m$ on functions of $\bz$ of Theorem \ref{actdem} above: in both cases, the action is by multiplication by $v^{\frac{mr}{2}}e_m(\bz)$. This involves writing the conserved quantity at $n\to\infty$ in both cases.

Second, if we use the expression of the conserved quantity $C_m$ (resp. ${\mathcal C}_m$) as a function of 
$\Q_{\al,0},\Q_{\al,1}$ 
(resp. ${\mathcal D}_{\al,0},{\mathcal D}_{\al,1}$), we obtain the exact {\it same} combinations of shift operators.

This shows that the difference equations obeyed by \eqref{defGkn} and \eqref{Gfound} at level $k=1$ are identical. 
To complete the analysis, we should in principle examine the initial conditions. 
We have seen that $G_\bn^{(1)}=0$ as soon as any of the $n^{(\al)}$
are equal to $-1$. Let us now show that these conditions are not necessary to fix the solution, as each such term
comes with a vanishing prefactor, and therefore drops out of the difference equation.

This fact relies on an important result of Ref. \cite{qKR}, which was instrumental in proving the polynomiality property for the associated quantum cluster algebra. It relies on the Laurent polynomiality property which asserts that any
cluster variable may be expressed as a Laurent polynomial of any seed variables. The following Lemma was derived by combining the Laurent property of the quantum cluster algebra
for initial data ${\mathcal S}_0=\{ \Q_{\al,0},\Q_{\al,1}\}$ as well as for initial data 
${\mathcal S}_{-1}=\{ \Q_{\al,-1},\Q_{\al,0}\}$.

\begin{lemma}{(\cite{qKR}, Lemma 5.9 and its proof.)}\label{uniqp}
For any polynomial $p$ of the $\{\Q_{\al,i}\}$ with coefficients in $\Z_v$, there exists a unique expression of the form:
\begin{equation}\label{formp} p=\sum_{A\cup B=[1,r];A\cap B=\emptyset; m_\al(A,B)\geq 0} \left(\prod_{\al\in A} \Q_{\al,-1}^{m_\al(A,B)}\right) \,
c^{A,B}_\bm (\{\Q_{\gamma,0}\})\, 
\left(\prod_{\beta\in B} \Q_{\beta,1}^{m_\beta(A,B)}\right)
\end{equation}
where the coefficients $c^{A,B}_\bm$ are Laurent polynomials of the variables $\{Q_{\gamma,0}\}$.
\end{lemma}

In other words, any occurrence of $\Q_{\al,1}^{-1}$ in the Laurent polynomial expression of 
$p$ may be replaced by a term $\Q_{\al,-1}$, for which coefficients remain Laurent polynomials of the variables $\{\Q_{\gamma,0}\}$. This powerful property can be applied to the conserved quantities as well. Indeed, each quantity $C_m$ of \eqref{conservedm}
is a Laurent polynomial of the initial data
${\mathcal S}_0$ as well as of ${\mathcal S}_{-1}$ depending on whether it is expressed at $n=0$ or $n=-1$. 
Repeating the argument leading to Lemma \ref{uniqp}, we also find that each $C_m$ may be expressed in a 
unique way in the form \eqref{formp}. Let us examine the expression of $C_m$ as a Laurent polynomial 
of the initial data ${\mathcal S}_0$ more closely. From the hard particle condition and the explicit form of 
$y_i(0)$ (\ref{yodd}-\ref{yeven}), 
we see that the terms  containing negative powers of $\Q_{\al,1}$ in $C_m$ must be of the form 
$c_{A,B}(\{\Q_{\gamma,0}\})\Big( \prod_{\al \in A} \Q_{\al,1}^{-1}\Big)\Big(\prod_{\beta \in B} \Q_{\beta,1}\Big)$, for some disjoint subsets
$A,B\subset [1,r]$,
as each particle is exclusive of its neighbors on the graph. Such
terms may be rewritten as $\Big(\prod_{\al \in A} \Q_{\al,-1}\Big)c_{A,B}'(\{\Q_{\gamma,0}\}) 
\Big(\prod_{\beta \in B} \Q_{\beta,1}\Big)$
according to the above. Now consider the level $1$ quantity
$G_\bn^{(1)}=\phi\left(\Big(\prod_\beta \Q_{\beta,1}\Big)\Big( \prod_\al \Q_{\al,1}^{n^{(\al)}}\Big) \tau(z)\right)$
and insert $C_m$ as before. We get:
$$v^{\frac{mr}{2}}e_m(\bz)G_\bn^{(1)}= \phi\left(\left(\prod_{\beta=1}^r \Q_{\beta,1}\right)
\left(\prod_{\al=1}^r \Q_{\al,1}^{n^{(\al)}}\right)C_m\, \tau(z)\right)$$
Suppose some $n^{(\al)}=0$. The insertion of $C_m$, expressed in terms of ${\mathcal S}_0$ variables, will introduce terms of the form  $G_\bn^{(1)}$ with $n^{(\al)}=-1$, whenever $\Q_{\al,1}^{-1}$ occurs in $C_m$. These are precisely the unwanted terms, for which we showed that $G_\bn^{(1)}=0$. However, we need not impose this condition. Indeed, by the above argument we may replace the terms with $\Q_{\al,1}^{-1}$ in $C_m$ with $\Q_{\al,-1}$, up to a change of coefficient $c_{A,B}\to c_{A,B}'$. This gives a contribution of the form:
\begin{eqnarray*}&&\phi\left(\Big(\prod_\beta \Q_{\beta,1}\Big)\Big( \prod_{\gamma\neq\al} 
\Q_{\gamma,1}^{n^{(\gamma)}}\Big)\Q_{\al,-1}p\, \tau(z)\right)\\
&&\qquad \qquad =v^{-2\sum_{\gamma\neq \al} \Lambda_{\al,\gamma}n^{(\gamma)}}
\phi\left(\Big(\prod_\beta \Q_{\beta,1}\Big)\Q_{\al,-1} \Big(\prod_{\gamma\neq\al} 
\Q_{\gamma,1}^{n^{(\gamma)}}\Big)p\, \tau(z)\right)=0
\end{eqnarray*}
by first using the commutation relation \eqref{commualbet}, and then noting that $\Q_{\al,1}\Q_{\al,-1}=v^{-\Lambda_{\al,\al}}(\Q_{\al,0}^2-\Q_{\al+1,0}\Q_{\al-1,0})$ causes the evaluation to vanish. 
Hence the terms which would have created $G_\bn^{(1)}$ with $n^{(\al)}=-1$ 
drop from the equation. 

This phenomenon is examplified in the expressions of Examples \ref{oneexG} and \ref{twoexG}
showing the difference equations for respectively ${\mathfrak sl}_2$ \eqref{levonG}
(where the coefficient of the unwanted term vanishes for $n=0$), and for ${\mathfrak sl}_3$ 
(\ref{levonethreeone}-\ref{levonethreetwo}) (where the coefficients of the unwanted terms vanish when $n=0$ 
or $p=0$).

The same holds for the difference equations satisfied by \eqref{Gfound} at $k=1$. To prove it, we repeat the above argument, and note that unwanted terms from ${\mathcal C}_m$ take the form 
$$\pi(p^*)\left(\prod_{\gamma\neq\al} {\mathcal D}_{\gamma,1}^{n^{(\gamma)}}\right)
{\mathcal D}_{\al,-1}\, 1=0$$
for some polynomials $p^*$ of the $\{\Q_{\al,i}^*\}$. This is
due to the fact that ${\mathcal D}_{\al,-1}\, 1$=0. 
This latter property is a consequence of the following lemma, proved in Appendix 
B below, and of its immediate corollary.

\begin{lemma}\label{lemvan}
For any $\al\in [1,r]$, we have the following identity:
\begin{equation}
\sum_{I\subset [1,r+1]\atop |I|=\al} (z_I)^p a_{I}(\bz) =\left\{ \begin{matrix} 1 & {\rm if}\ p=0 \\
0 & {\rm for} \ p=-1,-2,...,\al-r-1\, .\end{matrix}\right. 
\end{equation}
\end{lemma}
This implies immediately the following:
\begin{cor}\label{corus}
We have ${\mathcal D}_{\al,-p}\, 1=0$ for all $p=1,2,...,r+1-\al$, and
${\mathcal D}_{\al,0}\, 1=v^{-\sum_\beta \Lambda_{\al,\beta}}$.
\end{cor}

The only initial data needed to feed the level $1$ difference equations is therefore $G_0^{(1)}=1$, and the solution is uniquely determined by the equations. The corresponding function \eqref{Gfound} for $\bn=0$ is also trivially equal to $1$, and Theorem \ref{raisingthm} follows in the level $1$ case.

\subsubsection{Proof for general level $k\geq 2$}

Let $V=\Z_v[\bz]^{S_{r+1}}$, the space of symmetric polynomials in $\bz$ with coefficients in $\Z_v$. 
Using the map $\phi$ of Definition \ref{phidef}, we construct the map $\Psi$ from $A_+$, the space of polynomials in $\Q_{\al,k}$ with coefficients in $\Z_v$, to $V$ as follows.
\begin{defn}\label{psidef}
For all $p\in \Z_v[\{Q_{\al,k}\vert \al\in [1,r],k\geq 1\}]$, we define:
\begin{equation}\label{defpsi}
\Psi(p):=\phi\left(\prod_{\beta=1}^r \Q_{\beta,1}\, p \,\tau(z)\right) 
\end{equation}
\end{defn}
In particular, this allows to rewrite \eqref{defGkn} as:
\begin{equation}G_{\bn}(q^{-1},\bz)=\Psi\left(\prod \Q_{\al,i}^{n_i^{(\al)}}\right)\end{equation}
and we have the normalization condition $\Psi(1)=1$.

Let $V_0$ denote the image of $A_+$ under $\Psi$. $V_0$ is a right module over $A_+^{\rm op}$ where 
the superscript $\rm op$ denotes the opposite multiplication, under the action:
$$\Q_{\al,k} \circ \Psi(p) = \Psi(p \Q_{\al,k})$$

\begin{thm}\label{repsthm} The operators ${\mathcal D}_{\al,k}$ act on $V_0$ by left-multiplication, 
and form a representation of the action of $A_+^{\rm op}$ on $V_0$, such that: 
$ \Q_{\al,k} \circ \Psi(p) ={\mathcal D}_{\al,k} \, \Psi(p) $.
\end{thm}
\begin{proof}
We use the anti-homomorphism $*$ that maps $\Q_{\al,k}\mapsto \Q_{\al,k}^*$ and reverses
the order of multiplication, while preserving $v$, and compose it with the representation $\pi$. 
To any polynomial $p$ of the $\{\Q_{\al,i}\}$ with coefficients in 
$\Z_v$ we associate the polynomial $p^*$ of the $\Q_{\al,i}^*$ by 
$p^*(\{\Q_{\al,i}^*\})= p(\{\Q_{\al,i}\})^*$, and finally 
$\pi(p^*)$ by the substitution $\Q_{\al,i}^*\to {\mathcal D}_{\al,i}$,
namely $\pi(p^*)=p^*(\{{\mathcal D}_{\al,i}\})$. 
We wish to prove that $\Psi(p)=\pi(p^*)\, 1$.
By Lemma \ref{uniqp} we may write:
$$ p=\sum_{A\cup B=[1,r];A\cap B=\emptyset; m_\al(A,B)\geq 0} \left(\prod_{\al\in A} \Q_{\al,-1}^{m_\al(A,B)}\right) \,
c^{A,B}_\bm (\{\Q_{\gamma,0}\})\, 
\left(\prod_{\beta\in B} \Q_{\beta,1}^{m_\beta(A,B)}\right)$$
where the coefficients $c^{A,B}_\bm$ are Laurent polynomials of the $\{\Q_{\gamma,0}\}$,
for any polynomial $p$ of the $\{\Q_{\al,i}\}$ obeying the quantum $Q$-system relations. 
As moreover $\Psi(\Q_{\al,-1} f)=0$ for any polynomial $f$, we see that
$$\Psi(p)= \Psi(\varphi(p)) $$
where
$$ \varphi(p)=\sum_{m_\al(\emptyset,[1,r])\geq 0} 
c^{\emptyset,[1,r]}_\bm (\{\Q_{\gamma,0}\})\, \prod_{\al=1}^r \Q_{\al,1}^{m_\al(\emptyset,[1,r])} $$
The map $\varphi$ is simply the truncation to the polynomial part of $p$ in the variables $\Q_{\al,1}$.
Let $\varphi^*$ the corresponding truncation of any Laurent polynomial of
$\{\Q_{\al,0}^*,\Q_{\al,1}^*\}$ to it polynomial part in $\{\Q_{\al,1}^*\}$.
We have
$$ \pi(p^*)\, 1= \varphi^*(\pi(p^*))\, 1$$
where we have used $\varphi(p)^*=\varphi^*(p^*)$, and $\pi(f^* ){\mathcal D}_{\al,-1}\, 1=0$ (by Corollary \ref{corus}) for all polynomials $f^*$ of the 
$\Q_{\al,i}^*$. By definition of $\Psi$ and $\phi$ and the evaluation $ev_0$, we may now 
evaluate $\varphi(p)$ at 
$\Q_{\al,0}=v^{-\sum_\beta \Lambda_{\al,\beta}}$ without altering $\Psi(p)= \Psi(ev_0(\varphi(p)))$. Note that
$ev_0(f)^*=ev_0^*(f^*)$ where $ev_0^*$ is the right evaluation at 
$\Q_{\al,0}^*=v^{-\sum_\beta \Lambda_{\al,\beta}}$ (after the dual normal ordering that puts all 
$\Q_{\al,0}^*$ to the right).
Finally, from Corollary \ref{corus} we have: 
$ \pi(p^*)\, 1=ev_0^*( \varphi(\pi(p^*)))\, 1$.
The two polynomials
$ev_0(\varphi(p))$ and $ev_0^*(\varphi(\pi(p^*)))$ are the {\it same} polynomial of respectively $\{Q_{\al,1}\}$
and $\{ {\mathcal D}_{\al,1}\}$ with coefficients in $\Z_v$.
Therefore, the statement $\Psi(p)=\pi(p^*)\, 1$
needs only be proved for a polynomial $p\in \Z_v[\{\Q_{\al,1},\al\in [1,r]\}]$, and in fact for any monomial of the form
$\prod_\al \Q_{\al,1}^{m_\al}$ with $m_\al\geq 0$.
This is exactly the level $1$ case of Theorem \ref{raisingthm}, which was proved in Sect. \ref{leveoneproof}
above. The Theorem follows
by using the anti-homomorphism property $\pi((p \Q_{\al,k})^*)={\mathcal D}_{\al,k} \pi(p^*)$.
\end{proof}

Finally, noting that $\Psi(1)=1$, and applying Theorem \ref{repsthm} iteratively, leads straightforwardly to 
Theorem \ref{raisingthm} for arbitrary level $k$.

\subsection{Level one case and degenerate Macdonald polynomials}

When restricted to level $1$, the formula of Corollary \ref{gracor} for graded characters reduces 
to the following, for $\bn=\{n^{(\al)}\}_{\al\in [1,r]}$:
\begin{equation}\label{gracorone}
\chi_\bn(q^{-1},\bz)=q^{-\frac{1}{2} \sum_{\al,\beta} n^{(\al)}{\rm Min}(\al,\beta)n^{(\beta)}+\frac{1}{2}\sum_{\al} \al n^{(\al)}} \,\prod_{\al=1}^r(M_{\al,1})^{n^{(\al)}}\, 1
\end{equation}
with $M_{\al,1}$ as in \eqref{otmacdo}.
We have the following:

\begin{thm}\label{eigenchar}
The level one ${\mathfrak sl}_{r+1}$ graded characters \eqref{gracorone} are eigenfunctions of the degenerate Macdonald difference operators $M_{\al,0}=M_\al$ of \eqref{degmac}, namely:
$$ M_{\al,0}\, \chi_\bn(q^{-1},\bz) = E_{\al,\bn}\, \chi_\bn(q^{-1},\bz), \quad 
E_{\al,\bn}=q^{\sum_{\beta} {\rm Min}(\al,\beta)n^{(\beta)}} $$
\end{thm}
\begin{proof}
Starting from formula \eqref{gracorone}, we compute:
\begin{eqnarray*}
M_{\al,0}\, \chi_\bn(q^{-1},\bz) &=&q^{-\frac{1}{2} \sum_{\al,\beta} n^{(\al)}{\rm Min}(\al,\beta)n^{(\beta)}+\frac{1}{2}\sum_{\al} \al n^{(\al)}} M_{\al,0}\,\prod_{\beta=1}^r(M_{\beta,1})^{n^{(\beta)}}\, 1\\
&=& q^{\sum_{\beta} {\rm Min}(\al,\beta)n^{(\beta)}}\,
\chi_\bn(q^{-1},\bz)
\end{eqnarray*}
by use of the commutation relations \eqref{commutdemnew}, and the fact that $M_{\al,0}\, 1=1$ by Lemma
\ref{lemvan}.
\end{proof}

Recall that the symmetric $A_r$ $(q,t)$-Macdonald polynomials
$P_\lambda^{q,t}(\bz)$ of the variables $\bz=(z_1,...,z_{r+1})$,
indexed by partitions $\lambda=(\lambda_1\geq \lambda_2\geq \cdots \geq \lambda_{r+1}\geq 0)$,
are defined as the unique family of common eigenvectors to the difference operators $M_\al^{q,t}$, $\al=1,2,...,r$,
and whose leading
term is the symmetric monomial $m_\lambda=\prod_i z_i^{\lambda_i}+{\rm permutations}$. The Macdonald 
polynomials $P_\lambda^{q,t}(\bz)$ satisfy the following duality property \cite{macdo}:
\begin{equation}\label{duamac}
P_\lambda^{q,t}(\bz)=P_\lambda^{q^{-1},t^{-1}}(\bz)
\end{equation}

Comparing this with the result of Theorem \ref{eigenchar}, we conclude:

\begin{cor}
The level one $A_r$ graded characters $\chi_\bn(q^{-1},\bz)$ are the following degenerate limits of the Macdonald polynomials:
\begin{equation}\label{wittalimit}
\chi_\bn(q^{-1},\bz)=\lim_{t\to\infty} P_\lambda^{q,t}(\bz)=P_\lambda^{q^{-1},0}(\bz)
\end{equation}
where the correspondence between $\bn$ and $\lambda$ is via:
$$ \lambda_1=n^{(1)}+n^{(2)}+\cdots +n^{(r)},\quad \lambda_2=n^{(2)}+\cdots +n^{(r)},\quad \ldots\quad  \lambda_{r}=n^{(r)},
\quad \lambda_{r+1}=0$$
\end{cor}

Note that we have picked $\lambda_{r+1}=0$, as the variables $\bz$ satisfy $z_1z_2\cdots z_{r+1}=1$, so that
$P_{\lambda_1,..,\lambda_r,\lambda_{r+1}}^{q,t}(\bz)
=(z_1z_2\cdots z_{r+1})^{\lambda_{r+1}}P_{\lambda_1,..,\lambda_r,0}^{q,t}(\bz)$ is independent of $\lambda_{r+1}$.

\begin{remark}
From eq.\eqref{wittalimit}, we may identify the graded level one character $\chi_\bn(q,\bz)$ with the Whittaker limit $t\to 0$
of the Macdonald polynomial $\lim_{t\to 0} P_\lambda^{q,t}(\bz)$. This shows in particular
that $\chi_\bn(q,\bz)$ is a polynomial of $q$.
\end{remark}

\begin{remark}
The raising operators $M_{\al,1}$ coincide with the raising operators $K_\al^+$ for Macdonald polynomials introduced by Kirillov and Noumi \cite{kinoum}, in the limit $t\to \infty$, as well as with the dual raising operators $K_\al^-$ in the Whittaker limit $t\to 0$.
\end{remark}

\section{Conclusion}\label{conclusion}

In this paper we have used the constant term identity \cite{qKR} for graded tensor product multiplicities involving solutions of the 
$A_r$ quantum $Q$-system to: (1) derive difference equations for the corresponding graded characters, and (2) write 
expressions for the graded characters in terms of generalized degenerate Macdonald $q$-difference operators,
which form a representation of the dual quantum $Q$-system.
This latter construction establishes in the case of $\sl_{r+1}$ an intriguing bridge between two standard mathematical 
theories: on one hand that of $A_r$ Macdonald operators and polynomials, and on the other hand that of quantum cluster 
algebras, specifically that of the $A_r$ quantum $Q$-system. 

Our $q$-difference operators coincide in the initial cluster $\{M_{\al,0},M_{\al,1}\}$ to respectively the Macdonald
operators and the Kirillov-Noumi raising operators for Macdonald polynomials \cite{kinoum}, both in the 
dual Whittaker limit $t\to \infty$. We may view the finite $t$ case as a deformation of our initial cluster.
The algebraic framework of Macdonald theory is the Double Affine Hecke Algebra (DAHA) \cite{Cheredbook}. 
In a forthcoming 
paper \cite{DFKdaha}
we define natural $t$-deformations of our $q$-difference operators in the context of DAHA, that reduce to
$M_{\al,n}$ \eqref{otmacdo} in the limit $t\to \infty$. The DAHA relations are in a sense the natural $t$-deformation of the quantum 
$Q$-system. This should extend to other types than $A_r$ as well, for which both quantum $Q$-systems and DAHA structures
are known.

Finally, the explicit representation of graded characters as iterated action of $q$-difference operators on the constant $1$
may be useful to explore the so-called conformal limit, in which some $n_i^{(\al)}$ are taken to infinity
(infinite tensor products, see e.g. \cite{FFinfinite} for the case of $\sl_2$).

\begin{appendix}
\section{Proof of Lemmas \ref{firstlemma} and \ref{secondlemma}}

The proof of Lemmas \ref{firstlemma} and \ref{secondlemma} goes as follows. 
First we rewrite the statement of the Lemmas as a vanishing condition for
the antisymmetrized version of some rational fraction of the $z$'s. Then we show that all residues at the poles
of this antisymmetrized expression vanish. Finally we conclude that the result is proportional to
the antisymmetrization of a polynomial of the $z$'s with a too small degree, which must therefore vanish.

\subsection{Antisymmetrization: general properties}\label{gendefsA}

For any function $f(\bz)$ of the variables $\bz=(z_1,...,z_N)$ we define the 
symmetrization ($S$) and antisymmetrization ($AS$) operators as:
\begin{eqnarray}
S(f)(\bz)&=&\frac{1}{N!} \sum_{\sigma \in S_N} f(z_{\sigma(1)},...,z_{\sigma(N)})\\
AS(f)(\bz)&=&\frac{1}{N!} \sum_{\sigma \in S_N} {\rm sgn}(\sigma)\, f(z_{\sigma(1)},...,z_{\sigma(N)})
\end{eqnarray}

We have the following immediate result:
\begin{lemma}\label{symlem}
For any function $f_{I,J}(\bz)$ of $\bz$ indexed by two subsets $I,J$ of $[1,N]$ we have:
$$\sum_{I,J\subset [1,N], I\cap J=\emptyset\atop
|I|=a,\ |J|=N-a} f_{I,J}(\bz) ={N\choose a}\, S\left( f_{I_0,J_0}(\bz) \right) $$
where $I_0=[1,a]$ and $J_0=[a+1,N]$.
\end{lemma}

Lemma \ref{symlem} allows to rephrase the statements of Lemmas \ref{firstlemma} and \ref{secondlemma}
as identities on symmetrized expressions.

For $\bz=(z_1,...,z_N)$, we define the Vandermonde determinant $\Delta(\bz)=\prod_{1\leq i<j\leq N}(z_i-z_j)$
It is anti-symmetric, hence $AS(\Delta(\bz))=\Delta(\bz)$, and moreover for any function $f(\bz)$ we have
$AS(\Delta(\bz)f(\bz))=\Delta(\bz)S(f(\bz))$. 

We have the following standard fact about anti-symmetric polynomials.
\begin{lemma}\label{aslem}
The non-zero anti-symmetric polynomial $P$ of $\bz$ of smallest total degree, namely such that 
$AS(P)=P$, is proportional to the Vandermonde determinant of the $z$'s, up to a 
constant independent of the $z$'s. 
\end{lemma}
This implies the following:
\begin{cor}\label{ascor}
For any polynomial $P(\bz)$ of total degree strictly less than $N(N-1)/2$,
we have $AS(P)=0$.
\end{cor}

\subsection{Proof of Lemma \ref{firstlemma}}

For any integers $b\geq a\geq 0$, and $p\geq m\geq 0$, $I_0=[1,a]$, $J_0=[a+1,a+b]$, let us define
$$\varphi_{a,b}^{m,p}(\bz)=S\left( z_{I_0}^m z_{J_0}^p \prod_{i\in I_0\atop j\in J_0} 
\frac{z_i}{z_i-z_j}\frac{z_j}{z_j-qz_i}\right)$$
We note that as $m\leq p$, then:
$$\varphi_{a,b}^{m,p}(\bz)=(z_1\cdots z_{a+b})^m \, \varphi_{a,b}^{0,p-m}(\bz) $$
We also define:
\begin{equation}\label{psidefi}
\psi_{a,b}^{m,p}(\bz):=\varphi_{a,b}^{m,p}(\bz) -q^{a (p-m)} \varphi_{b,a}^{p,m}(\bz)
\end{equation}
Using Lemma \ref{symlem}, it is straightforward to show that the statement of Lemma \ref{firstlemma} is equivalent to:
\begin{equation}\label{refirst}
\psi_{a,b}^{m,p}(\bz)=0
\end{equation}

In the following, we use the notation $\Delta_I=\prod_{1\leq k\leq \ell \leq n} (z_{i_k}-z_{i_\ell})$
for any {\it ordered} set $I=\{i_1,i_2,...,i_n\}$.
With $I_0,J_0$ as above, we now express:
\begin{eqnarray*}
\Delta(\bz) \varphi_{a,b}^{m,p}(\bz)  &=& AS\left(\Delta(\bz) z_{I_0}^m z_{J_0}^p \prod_{i\in I_0\atop j\in J_0} 
\frac{z_i}{z_i-z_j}\frac{z_j}{z_j-qz_i}\right) \\
&=& AS\left(\Delta_{I_0} \Delta_{J_0} z_{I_0}^{m+b} z_{J_0}^p \prod_{i\in I_0\atop j\in J_0} 
\frac{z_j}{z_j-qz_i}\right) 
\end{eqnarray*}

The only possible poles of $\Delta(\bz) \varphi_{a,b}^{p,q}(\bz)$ are for 
$z_i\to q z_j$ for $i\neq j$. Let us compute the residue
at the pole $z_2\to q z_1$ in $z_2$. Pick two ordered sets $I_0',J_0'$ with $I_0'\cap J_0'=\emptyset$, 
$I_0'\cup J_0'=[1,a+b]$, $|I_0'|=a$, $|J_0'|=b$, and such that
$1$ is the first element of $I'_0=\{1\}\cup I_1$ and $2$ the last element of $J_0'=J_1\cup \{2\}$. 
For any subset $L\subset [1,N]$, we denote by $AS_L$ the antisymmetrization over the set $\{z_i\}_{i\in L}$.
We compute
\begin{eqnarray*}
{\rm Res}_{z_2\to qz_1} \Delta(\bz) \varphi_{a,b}^{m,p}(\bz)  &=& z_1^{m+b} (q z_1)^{p+1} AS_{[3,a+b]}\left( 
\Delta_{I_1} \prod_{i\in I_1}(z_1-z_i) \Delta_{J_1} \prod_{j\in J_1}(z_j-q z_1) \right. \\
&&\qquad \left. 
z_{I_1}^{m+b} z_{J_1}^p \prod_{i\in I_1\atop j\in J_1} 
\frac{z_j}{z_j-qz_i}  \prod_{i\in I_1}\frac{q z_1}{q z_1-qz_i} \prod_{ j\in J_1} \frac{z_j}{z_j-qz_1} \right)\\
&=& q^{p+1}z_1^{m+p+b+a}  \Delta(\bz') \varphi_{a-1,b-1}^{m+1,p+1}(\bz')
\end{eqnarray*}
where $\bz'=(z_3,...,z_{a+b})$. Using \eqref{psidefi}, we deduce that:
\begin{eqnarray*}
{\rm Res}_{z_2\to qz_1} \Delta(\bz) \psi_{a,b}^{m,p}(\bz)&=&
\Delta(\bz') \left\{q^{p+1}z_1^{m+p+b+a}  \varphi_{a-1,b-1}^{m+1,p+1}(\bz')
-q^{a(p-m)} q^{m+1} z_1^{m+p+b+a}\varphi_{b-1,a-1}^{p+1,m+1}(\bz')\right\} \\
&=& q^{p+1}z_1^{m+p+b+a} \Delta(\bz')\psi_{a-1,b-1}^{m+1,p+1}(\bz') 
\end{eqnarray*}

We now proceed by induction on $a$. For $a=0$, we have:
$$ \varphi_{0,b}^{m,p}(\bz)=S\left((z_1\cdots z_b)^p\right)=(z_1\cdots z_b)^p=\varphi_{b,0}^{p,m}(\bz) $$
hence $\psi_{0,b}^{m,p}(\bz)=0$.
Assuming that $\psi_{a-1,b-1}^{m+1,p+1}(\bz') =0$, we see that the residue at $z_2\to q z_1$ of
$\psi_{a,b}^{m,p}(\bz)$ vanishes, hence the is no pole of the form $1/(z_2-qz_1)$ in the antisymmetrized expression.
By symmetry, this holds for any pole $z_i\to q z_j$. We conclude that $\psi_{a,b}^{m,p}(\bz)$ is a polynomial.
Using the antisymmetrization formula, we easily get:
\begin{eqnarray*}\Delta(\bz) \varphi_{a,b}^{m,p}(\bz) &=&
AS\left(\Delta_{I_0} \Delta_{J_0} z_{I_0}^{m+b} z_{J_0}^{p+a} \prod_{i\in I_0\atop j\in J_0} 
\frac{1}{z_j-qz_i}\right) \\
&=& (z_1\cdots z_{a+b})^{p+a} 
AS\left(\Delta_{I_0} \Delta_{J_0} z_{I_0}^{b-a-(p-m)} \prod_{i\in I_0\atop j\in J_0} 
\frac{1}{z_j-qz_i}\right) 
\end{eqnarray*}
Similarly:
$$\Delta(\bz) \varphi_{b,a}^{p,m}(\bz)=
 (z_1\cdots z_{a+b})^{p+a} 
AS\left(\Delta_{I_0} \Delta_{J_0} z_{I_0}^{b-a-(p-m)} \prod_{i\in J_0\atop j\in I_0} 
\frac{1}{z_j-qz_i}\right) $$
Finally, we have:
$$\frac{\Delta(\bz)\psi_{a,b}^{m,p}(\bz)}{ (z_1\cdots z_{a+b})^{p+a-1}}=
AS\left(\Delta_{I_0} \Delta_{J_0} z_{I_0}^{b-a+1-(p-m)}z_{J_0}\left\{ \prod_{i\in I_0\atop j\in J_0} 
\frac{1}{z_j-qz_i}-q^{a(p-m)} \prod_{i\in J_0\atop j\in I_0} 
\frac{1}{z_j-qz_i}\right\} \right)$$
where the r.h.s. is a polynomial, as $b-a+1-(p-m)\geq 0$ and it has no poles at $z_i=qz_j$. Writing $N=a+b$,
its total degree is:
$$\frac{N(N-1)}{2}+ma+pb-N(p+a-1)=\frac{N(N-1)}{2}-(p-m)a -N(a-1)<\frac{N(N-1)}{2}$$
for $a\geq 1$.
The degree of the polynomial is therefore too small, 
and it must vanish by Corollary \ref{ascor}. The Lemma \ref{firstlemma} follows.

\subsection{Proof of Lemma \ref{secondlemma}}

We proceed analogously. 
For $I_0=[1,a]$ and $J_0=[a+1,2a]$, we define
$$\theta_a(\bz)=S\left( \prod_{i\in I_0\atop j\in J_0} \frac{z_i}{z_i-z_j}\frac{z_j}{z_j-qz_i}
 \left( 1-q^a \frac{z_{I_0}}{z_{J_0}} \right) \right)$$
We also have:
$$\Delta(\bz)\theta_a(\bz)=AS\left(\Delta_{I_0} \Delta_{J_0} z_{I_0}^a  
\prod_{i\in I_0\atop j\in J_0}\frac{z_j}{z_j-qz_i}
 \left( 1-q^a \frac{z_{I_0}}{z_{J_0}} \right) \right)$$
 Let us compute the residue of the pole of this expression at $z_2\to q z_1$.
 As before, we pick two ordered sets $I_0'$ and $J_0'$ of cardinality $a$ such that $I_0'\cap J_0'=\emptyset$,
 $I_0'\cup J_0'=[1,2a]$ and $1$ is the first element of $I_0'=\{1\}\cup I_1$ and $2$ the last element of
 $J_0'=J_1\cup \{2\}$. We compute:
 \begin{eqnarray*}
 {\rm Res}_{z_2\to qz_1} \Delta(\bz) \theta_{a}(\bz)&=& q z_1^{a+1}
 AS\left(\Delta_{I_1} \prod_{i\in I_1}(z_1-z_i)  \Delta_{J_1}\prod_{j\in J_1}(z_j-q z_1) z_{I_1}^a  \right. \\
&&\quad \left. \prod_{i\in I_1\atop j\in J_1}\frac{z_j}{z_j-qz_i} \prod_{i\in I_1}\frac{q z_1}{q z_1-q z_i} 
\prod_{j\in J_1}\frac{z_j}{z_j-q z_1}
 \left( 1-q^{a-1} \frac{z_{I_1}}{z_{J_1}} \right) \right)\\
 &=&  q z_1^{2a} (z_3\cdots z_{2a})
 AS\left(\Delta_{I_1}  \Delta_{J_1} z_{I_1}^{a-1}\prod_{i\in I_1\atop j\in J_1}\frac{z_j}{z_j-qz_i} 
 \left( 1-q^{a-1} \frac{z_{I_1}}{z_{J_1}} \right) \right)\\
 &=&  q z_1^{2a} (z_3\cdots z_{2a}) \Delta(\bz') \theta_{a-1}(\bz')
 \end{eqnarray*}
where we denote by $\bz'=(z_3,z_4,...,z_{2a})$. 

Likewise, we define for $I_2=[1,a+1]$ and $J_2=[a+2,2a]$:
$$\varphi_a(\bz)= S\left( \prod_{i\in I_2\atop j\in J_2} \frac{z_i}{z_i-z_j}\frac{z_j}{z_j-q z_i}  \right) $$
We also have:
$$\Delta(\bz)\varphi_a(\bz)= AS\left(\Delta_{I_2}\Delta_{J_2} z_{I_2}^{a-1} 
\prod_{i\in I_2\atop j\in J_2} \frac{z_j}{z_j-q z_i}  \right)$$
Let us compute the residue of the pole of this expressions at $z_2\to q z_1$.
We pick two ordered sets $I_2'$ and $J_2'$ such that $I_2'\cap J_2'=\emptyset$,
 $I_2'\cup J_2'=[1,2a]$, $|I_2'|=a+1$, $|J_2'|=a-1$, 
 and $1$ is the first element of $I_2'=\{1\}\cup I_3$ and $2$ the last element of
 $J_2'=J_3\cup \{2\}$. We compute:
 \begin{eqnarray*}
 {\rm Res}_{z_2\to qz_1} \Delta(\bz) \varphi_{a}(\bz)&=&q z_1^{a}
 AS\left(\Delta_{I_3} \prod_{i\in I_3}(z_1-z_i)  \Delta_{J_3}\prod_{j\in J_3}(z_j-q z_1) z_{I_3}^{a-1}   \right. \\
&&\left. \prod_{i\in I_3\atop j\in J_3} \frac{z_j}{z_j-q z_i}  \prod_{i\in I_3} \frac{q z_1}{q z_1-q z_i} 
 \prod_{j\in J_3}\frac{z_j}{z_j-q z_1}\right)\\
 &=& q z_1^{2a} AS\left(\Delta_{I_3}\Delta_{J_3} z_{I_3}^{a-1} z_{J_3} \prod_{i\in I_3\atop j\in J_3} \frac{z_j}{z_j-q z_i}
 \right)\\
 &=& q z_1^{2a} (z_3z_4\cdots z_{2a})\Delta(\bz')\varphi_{a-1}(\bz')
 \end{eqnarray*}
 We conclude that
$$ {\rm Res}_{z_2\to qz_1} \Delta(\bz) \left\{ \theta_{a}(\bz)-\varphi_{a}(\bz)\right\}=q z_1^{2a} (z_3z_4\cdots z_{2a})
\Delta(\bz')\left\{ \theta_{a-1}(\bz')-\varphi_{a-1}(\bz')\right\}$$

We proceed by induction on $a$. For $a=1$ we have 
$$(z_1-z_2)\theta_1(z_1,z_2)=AS\left( \frac{z_1z_2}{z_2-qz_1} (1-q \frac{z_1}{z_2})\right)=z_1-z_2 $$
Analogously, we find
$$\varphi_1(z_1,z_2)=S(1)=1$$
hence $\theta_1(z_1,z_2)-\varphi_1(z_1,z_2)=0$. Assuming that $\theta_{a-1}(\bz')-\varphi_{a-1}(\bz')=0$,
we deduce that $\Delta(\bz)(\theta_{a}(\bz)-\varphi_{a}(\bz))$ has no pole at $z_2=qz_1$. By symmetry, it has no pole
at any $z_i=qz_j$, hence it is a polynomial. Finally we write:
\begin{eqnarray*}\frac{\Delta(\bz)(\theta_{a}(\bz)-\varphi_{a}(\bz))}{(z_1z_2\cdots z_{2a})^{a-1}}&=&AS\left(\Delta_{I_0} \Delta_{J_0} z_{I_0}\prod_{i\in I_0\atop j\in J_0}\frac{1}{z_j-qz_i}
 \left( z_{J_0}-q^a z_{I_0} \right) \right)\\
 &&-AS\left(\Delta_{I_2} \Delta_{J_2} z_{J_2}^2\prod_{i\in I_2\atop j\in J_2}\frac{1}{z_j-qz_i}\right) 
 \end{eqnarray*}
 where the r.h.s. is a polynomial of total degree $N(N-1)/2-2a(a-1)<N(N-1)/2$ for $a\geq 2$ and $N=2a$. 
 By Corollary \ref{ascor},
 the result must vanish, and the Lemma \ref{secondlemma} follows.

\section{Proof of Lemma \ref{lemvan}}

Notations are as in Sect. \ref{gendefsA}.
Let us consider for $\al\in [1,N]$ and $p\in \Z$ the quantity
$$ A_{\al,p}(\bz)=\sum_{I\subset [1,N]\atop |I|=\al} (z_I)^p \, a_I(\bz) $$
Picking the particular subset $I_\al=\{1,2,...,\al\}$, we may also write
$$ A_{\al,p}(\bz)={N\choose \al} \, S\left( (z_{I_\al})^p \, a_{I_\al}(\bz)\right) $$
We now wish to eliminate the denominators in this (symmetric) expression. We 
use that $AS(\Delta(\bz)f(\bz))=\Delta(\bz) \, S\left(f(\bz)\right)$ for any $f$ to rewrite:
$$ \Delta(\bz)\, A_{\al,p}(\bz)={N\choose \al} AS\left((z_{I_\al})^{p+N-\al} \Delta(z_1,...,z_\al)
\Delta(z_{\al+1},...,z_N)  \right) $$
The function to be antisymmetrized is a polynomial if $p\geq \al-N$, and then it has total degree
$$ \al(p+N-\al)+\frac{\al(\al-1)}{2}+\frac{(N-\al)(N-\al-1)}{2}= \al\, p+\frac{N(N-1)}{2}$$
By Corollary \ref{ascor}, we deduce that for $p=-1,-2,...,\al-N$ the antisymmetrized 
expression must vanish.

When $p=0$, the degree is exactly $N(N-1)/2$ and therefore $\Delta(\bz)\, A_{\al,p}(\bz)$
is proportional to $\Delta(\bz)$. The proportionality constant is fixed by evaluating $A_{\al,0}$
in the successive limits  $z_1\to\infty,z_2\to\infty,...,z_\al\to\infty$, and we finally get $A_{\al,0}=1$.

This completes the proof of Lemma \ref{lemvan}.

\end{appendix}
\bibliographystyle{alpha}

\bibliography{refs}

\end{document}